\documentclass[reqno]{amsart}
\usepackage[latin1]{inputenc}
\usepackage{amssymb}
\usepackage{amsfonts}
\usepackage{lipsum}
\usepackage[normalem]{ulem}
\usepackage{paralist}
\usepackage{enumitem}
\usepackage{indentfirst}
\usepackage{amsmath}
\usepackage[integrals]{wasysym}
\usepackage{amsthm}
\usepackage[all]{xy}
\usepackage{mathrsfs}
\usepackage{standalone}
\usepackage{color}
\usepackage{adforn}
\usepackage{comment}
\usepackage[bookmarks]{hyperref}
\usepackage{cleveref}
\usepackage{cancel}

\newcommand\blfootnote[1]{%
	\begingroup
	\renewcommand\thefootnote{}\footnote{#1}%
	\addtocounter{footnote}{-1}%
	\endgroup
}

\crefformat{section}{\S#2#1#3} 
\crefformat{subsection}{\S#2#1#3}
\crefformat{subsubsection}{\S#2#1#3}
\crefrangeformat{section}{\S\S#3#1#4 to~#5#2#6}
\crefmultiformat{section}{\S\S#2#1#3}{ and~#2#1#3}{, #2#1#3}{ and~#2#1#3}

\theoremstyle{cupplain}
\newtheorem{theorem}{Theorem}[section]
\newtheorem{lemma}[theorem]{Lemma}
\newtheorem{corollary}[theorem]{Corollary}
\newtheorem{steps}{Step}
\newtheorem{thm}{Theorem}
\newtheorem*{thm*}{Theorem}

\newtheorem*{Franks' lemma}{Franks' Lemma}
\newtheorem*{keylemma}{Key Obstruction Lemma}
\newtheorem*{claim*}{Claim}
\newtheorem{claim}{Claim}
\newtheorem{prop}[theorem]{Proposition}
\newtheorem*{invproplemma}{Lemma (Invariance property)}
\newtheorem*{anglelemma}{Lemma (Uniform angle)}
\newtheorem*{dominatelemma}{Lemma (Uniform domination property)}
\DeclareMathOperator{\im}{Im}
\DeclareMathOperator{\m}{\mathfrak{m}}
\theoremstyle{cupdefinition}
\newtheorem{definition}{Definition}[section]
\newtheorem{remark}[theorem]{Remark}
\newtheorem{example}[theorem]{Example}
\newtheorem*{acknowledgement}{Acknowledgement}

\numberwithin{equation}{section}

\begin{document}

\title[Robust transitivity for endomorphisms]{Robust transitivity and domination for endomorphisms displaying critical points}
 \author[C. Lizana]{C. Lizana} 
\address{Departamento de Matem\'{a}tica. Instituto de Matem\'{a}tica
e Estat\'{i}stica.
Universidade Federal da Bahia. Av. Adhemar de Barros s/n, 40170-110. 
Salvador, Bahia, Brazil.}
 \email{clizana@ufba.br}


\author[R. Potrie]{R. Potrie}
\address{CMAT, Facultad de Ciencias, Universidad de la Rep\'{u}blica, Uruguay.}
 \email{rpotrie@cmat.edu.uy}

\author[E. R. Pujals]{E. R. Pujals}
\address{CUNY, 365 Fifth Avenue New York, NY 10016, USA.}
 \email{epujals@gc.cuny.edu}

\author[W. Ranter]{W. Ranter} 
\address{Instituto de Matem\'{a}tica. Universidade Federal de Alagoas, Campus A.S. Simoes s/n, 57072-090. Macei\'o, Alagoas, Brazil.}
\email{wagnerranter@im.ufal.br}


\maketitle

\blfootnote{\P \, W.R. and \ddag \, R.P. were partially supported by CNPq/MCTI/FNDCT project 409198/2021-8, Brazil.  \ddag \, R.P. was also partially funded by CSIC. }\blfootnote{\dag \, C.L. and \S \, E.P. were partially supported by CNPq/MCTI/FNDCT project 406750/2021-1, Brazil. \S \, E.P. received support of NSF via grant DMS1956022.}

\begin{abstract}
	We show that robustly transitive endomorphisms of a closed manifolds must have a non-trivial dominated splitting or be a local diffeomorphism. This allows to get some topological obstructions for the existence of robustly transitive endomorphisms. To obtain the result we must understand the structure of the kernel of the differential and the recurrence to the critical set of the endomorphism after perturbation. 
\end{abstract}

\section{Introduction}
Throughout this paper, unless specified, $M$ denotes a $d$-dimensional closed Riemannian manifold and $\mathrm{End}^1(M)$ the set of all $C^1$-maps from $M$ into itself endowed with the $C^1$-topology. The elements of $\mathrm{End}^1(M)$ are called \textit{endomorphisms}. Some of them exhibit \textit{critical points}, that is, points on which the derivative is not a linear isomorphism; and the other ones, endomorphisms without critical points, are local diffeomorphisms or diffeomorphisms.

An endomorphism $f$ is said to be \textit{robustly transitive} if there exists a neighborhood $\mathcal{U}_f$ of $f$ in $\mathrm{End}^1(M)$ such that every $g \in \mathcal{U}_f$ is transitive, where \textit{transitive} means the existence of a dense forward orbit.

It should be pointed out that we are actually defining $C^1$ robust transitivity. $C^r$ robust transitivity could also be defined using the $C^r$-topology.  Our approach cannot be extended for $C^r$ robust transitivity since many of the techniques used here do not work in $C^r$-topology and, in \cite{IP}, it is constructed an example of a $C^2$-robustly transitive endomorphism which is not $C^1$-robustly transitive. 

The main purpose of this paper is to show that dominated splitting is a necessary condition for the  existence of robustly transitive endomorphisms displaying critical points. Concretely, we prove the following result.

\begin{thm}\label{thm A}
	Every robustly transitive endomorphism displaying critical points admits a nontrivial dominated splitting.
\end{thm}

An endomorphism $f$ admits nontrivial \textit{dominated splitting of index $\kappa$} if for every \textit{orbit} $(x_i)_i$, that is, a sequence of points $(x_i)_i$ in $M$ such that $f(x_{i})=x_{i+1}$ for each $i \in \mathbb{Z}$, there are two 
nontrivial families $(E(x_i))_i$ and $(F(x_i))_i$ of $\kappa$ and $(d-\kappa)$-dimensional subspaces, respectively, satisfying the following.

\begin{description}
	\item[Invariant splitting:] for each $i \in \mathbb{Z}$ one has that
	\begin{align*}
	& T_{x_i}M=E(x_i)\oplus F(x_i), \ \  Df(E(x_i)) \subseteq E(f(x_i))\\ & \text{and} \ \ Df(F(x_i))=F(f(x_i));
	\end{align*}
	
	\item[Domination property:] there is an integer $\ell>0$ independent of any orbit such that
	\begin{align*}
	\|Df^{\ell}(u)\|\leq \frac{1}{2}\|Df^{\ell}(v)\|,
	\end{align*}
	for every unit vectors $u \in E(x_i)$ and $v \in F(x_i)$.
\end{description}

We will sometimes abuse notation and call the families of subspaces by $E$ and $F$ subbundles (cf. Remark \ref{rem.subbundles}). Further, we denote the domination property by $E\prec F$ or $E\prec_{\ell} F$ if we want to emphasize the role of $\ell$.  See \S\ref{sec:weak-hyp} for further details about dominated splitting.

The authors believe that, in general, robustly transitive endomorphisms displaying critical points require more than just dominated splitting. It is feasible that the dominated splitting $E\oplus F$ provided by Theorem~\ref{thm A} admits the \textit{finest dominated splitting}\footnote{$E_1\oplus \cdots \oplus E_k$ is the finest dominated splitting if $E_1\prec E_2 \prec \cdots \prec E_k$ and none of the invariant subbundle $E_i$ admits a dominated splitting.} such as $E\oplus_{i=1}^k{F_i}$ which the derivative $Df$ restricted to the extremal subbundle $F_k$ is volume expanding. It was proved for surface endomorphism in \cite{CW1}. For higher dimension, that would be a similar result as one obtained for diffeomorphisms in \cite{BDP} which states, in \cite[Theorem 4]{BDP}, that every $C^1$-robustly transitive diffeomorphism admits a finest dominated splitting such that the derivative $Df$ restricted to the extremal subbundles are volume contracting and volume expanding, respectively. See \S\ref{brief} for further discussion. 


As a consequence of the main result, we obtain the following topological obstruction. The proof is in  \S\ref{sec:proof of thm A}.

\begin{corollary}\label{cor:topobst}
	Even dimensional spheres do not admit robustly transitive endomorphisms.
\end{corollary}

Note that the existence of robustly transitive diffeomorphisms in $S^3$ is a well known open problem, and a negative answer is expected (see e.g. \cite{DPU}). It makes sense to ask if examples of robustly transitive endomorphisms in $S^3$ may exist, while we expect this question to be difficult. 

We introduce now the following result that will be useful to prove Theorem~\ref{thm A}.

\begin{thm}\label{thm A1}
	Let $f_0$ be a robustly transitive endomorphism displaying critical points and  a neighborhood $\mathcal{U}_0$ of $f_0$ in $\mathrm{End}^1(M)$ such that every $f$ in $\mathcal{U}_0$ is transitive. Then, there exist an integer $\ell\geq 1$, a number $\alpha>0$, and subset $\mathcal{F}$ of $\mathcal{U}_0$ such that hold the following.
	\begin{enumerate}[label=$\mathrm{(\alph*)}$]
		\item $f_0$ is accumulated by the endomorphisms in $\mathcal{F}$; and
		\item every $f \in \mathcal{F}$ admits a dominated splitting $E\oplus F$ such that $E\prec_{\ell} F$ and the angle between $E$ and $F$ is greater than $\alpha$\footnote{It means that $\varangle(u,v)\geq \alpha$, for all vectors $u \in E(x_i),v \in F(x_i),$ for each $x_i$ along the orbit $(x_i)_i \in \Lambda_f$. For details see \S\ref{sec:weak-hyp}.}.
	\end{enumerate}
\end{thm}

Theorem~\ref{thm A1} implies uniformity of the dominated splitting for endomorphisms in $\mathcal{F}$ which will allow us, since $f_0$ is accumulated by $\mathcal{F}$, to extend the dominated splitting to $f_0$. 

\subsection{A brief history of robust transitivity}\label{brief}

Robust transitivity has been well studied in the diffeomorphism context. The first examples were given by Shub on $\mathbb{T}^4$ in \cite{Shub-ex} and by Ma\~{n}\'{e} on $\mathbb{T}^3$ in \cite{Mane}, which have an underlying structure weaker than hyperbolic, known as partially hyperbolic. However, Ma\~{n}\'{e} proved for surface diffeomorphisms in \cite{Mane-Closinglemma} that robust transitivity implies hyperbolicity and, in particular, the only surface that admits such systems is the torus, $\mathbb{T}^2$.

Bonatti and D\'{i}az, in \cite{BD}, construct a powerful geometric tool (called blender) to produce robustly transitive partially hyperbolic diffeomorphisms. Later, in \cite{BV}, Bonatti and Viana construct the first examples of robustly transitive diffeomorphisms with dominated splitting which are not partially hyperbolic. In \cite{DPU, BDP}, Bonatti, D\'{i}az, Pujals, and Ures prove for diffeomorphism on three and higher-dimensional manifolds that robust transitivity requires some weak form of hyperbolicity.

In view of this, a natural question arise. In general, do robustly transitive endomorphisms require some weak form of hyperbolicty?

In the local diffeomorphisms scenario, there are several advances. Based on the examples of robustly transitive diffeomorphisms were constructed robustly transitive non-expanding endomorphisms. In \cite{LP} were obtained necessary and sufficient conditions for robustly transitive local diffeomorphisms. In particular, it is shown that it is not necessary any weak form of hyperbolicity for the existence of a robustly transitive local diffeomorphism; a trivial example is an expanding linear endomorphism with complex eigenvalues, which does not admit a dominated splitting. That result shows a difference between diffeomorphisms and local diffeomorphisms setting. We note however that one can think of an endomorphism as having a strong stable bundle consisting on pre-orbits, and with this point of view, the results bear a closer analogy to those of diffeomorphisms. 

In the endomorphisms displaying critical points setting, the first examples were given in \cite{Berger-Rovella} and \cite{ILP}. Although these examples exhibit some form of weak hyperbolicity and are homotopic to a hyperbolic linear endomorphism on $\mathbb{T}^2$, it was recently that any result about necessary conditions were established. In \cite{CW1}, it was proved for surface endomorphisms that a weak form of hyperbolicity is needed for robust transitivity, so-called partial hyperbolicity. Furthermore, it was also proved that only the Torus and the Klein bottle support robustly transitive endomorphism exhibiting critical points; and that the action of such map on  the first homology group has at least an eigenvalue with modulus greater than one. 
Later, it was given in \cite{CW2} new classes of examples of robustly transitive endomorphisms. The examples are homotopic to expanding linear endomorphism on the torus or the Klein bottle; and an example of robustly transitive endomorphism of zero degree. In higher dimension, only recently in \cite{Morelli} it is constructed the first examples of robustly transitive endomorphisms displaying critical points.

\subsection{Comments about some previous approaches}

Here, we briefly comment on the main ingredients used to show that some (weak) form of hyperbolicity is a necessary condition for the existence of robust transitivity.


\subsubsection{Key obstructions for robust transitivity}

In a broad sense, an obstruction for robust transitivity is some phenomenon which is incompatible with this feature.

Here, we discuss about these phenomena which play an important role in the proof that robust transitivity requires some weak form of hyperbolicity. These phenomena are in some sense related to ``the candidate for dominated splitting". Let us  define them. 

\begin{itemize}
	\item \textit{Source:} a periodic point $p$ for $f$, where $n_p \geq 1$ denotes its period, such that $Df^{n_p}_p$ is a matrix having all the eigenvalues with modulus greater than one.
	
	\smallskip
	
	\item \textit{Sink:} a periodic point $p$ for $f$, where $n_p \geq 1$ denotes its period, such that $Df^{n_p}_p$ is a matrix having all the eigenvalues with modulus less than one.
\end{itemize}

The set of all critical points of an endomorphism $f$ will be denoted by $\mathrm{Cr}(f)$ and its interior in $M$ denoted by $\mathrm{int}(\mathrm{Cr}(f))$.

\begin{itemize}
	\item  \textit{Full-dimensional kernel:} there exist a point $x \in \mathrm{Cr}(f)$ and an integer $n\geq 1$ such that $\dim \ker(Df^n_x)=d$.
\end{itemize}

\subsubsection{Key obstructions vs. dominated splitting for diffeomorphisms} 

Here, we discuss the role of sources and sinks as obstructions to obtain that a robustly transitive diffeomorphism admits a dominated splitting. First, it is easy to see that transitive diffeomorphisms do not admit neither sources nor sinks, otherwise there is a small neighborhood such that its image by some (backward or forward) iterate goes into itself, which is incompatible with transitivity.

Let $\mathcal{U}_f$ be a neighborhood of $f$ in $\mathrm{End}^1(M)$, where all endomorphisms in $\mathcal{U}_f$ are transitive diffeomorphisms. Let us start commenting the approaches for surface diffeomorphisms in \cite{Mane-Closinglemma} and on higher-dimensional manifold in \cite{DPU} and \cite{BDP}. 

\begin{itemize}
	\item In \cite{Mane-Closinglemma} is used the fact that sources and sinks are obstructions for transitivity to prove that the set of all the periodic points of any surface diffeomorphism in $\mathcal{U}_f$ is hyperbolic, and so, it has a ``natural" splitting given by the stable and unstable directions. Later, it is proved that the lack of domination property allows to create, up to a perturbation, sinks or sources contradicting the robust transitivity. Finally, it is used the classical result so-called Closing Lemma to extend the dominated splitting to the whole surface. Consequently,  every robustly transitive surface diffeomorphism admits a dominated splitting (weak hyperbolicity). More precisely, it is proved the following dichotomy.
	
	\begin{thm*}[\cite{Mane-Closinglemma}] Let $M$ be a closed surface. Then there is a residual subset $\mathcal{R} \subseteq \mathrm{Diff}^1(M)$ (i.e., $\mathrm{Diff}^1(M)$ the set of all the diffeomorphisms), $\mathcal{R}=\mathcal{R}_1\sqcup \mathcal{R}_2,$ such that every $f \in \mathcal{R}_1$ is an Axiom A and every $f \in \mathcal{R}_2$ has infinitely many sources and sinks.
	\end{thm*}
	
	\noindent In particular, every robustly transitive surface diffeomorphism is an Anosov diffeomorphism.
\end{itemize}

Note that in higher-dimensional manifolds even if each periodic point is hyperbolic, they could have different indexes (i.e., unstable directions of different dimensions) which hamper the choice of a ``natural" splitting over the set of all the periodic points. Thus, the approach followed in \cite{DPU,BDP} was slightly different.

\begin{itemize}
	\item In this context, they consider a hyperbolic saddle point $p$ of the diffeomorphism $f$ and its homoclinic class, denoted by $H(p,f)$. Then, one define ``naturally" a splitting using the stable and unstable directions over $H(p,f)$ and prove that if such splitting is no dominated, one can create a source or a sink for some perturbation of $f$. Since any diffeomorphism in $\mathcal{U}_f$ admits neither sources nor sinks, one has that $H(p,f)$ admits a dominated splitting. Finally, using classical results such as the Closing Lemma and Connecting Lemma, they extend the splitting to the whole manifold, proving that robust transitivity for diffeomorphisms requires dominated splitting (weak hyperbolicity).
	
	In fact, the result above follows as a consequence of the following.
	
	\begin{thm*}[\cite{BDP}]
		Let $p$ a hyperbolic saddle of a diffeomorphism $f$ defined on $M$. Then,
		\begin{itemize}
			\item either the homoclinic class $H(p,f)$ of $p$ admits a dominated splitting;
			\item or given any neighborhood $U$ of $H(p,f)$ and any integer $\ell\geq 1$, there exists $g$ arbitrarily $C^1$-close to $f$ having $\ell$ sources or sinks arbitrarily close to $p$, whose orbits are contained in $U$.
		\end{itemize}
	\end{thm*}
	
	Even more,  it was proved that every robustly transitive diffeomorphism is volume hyperbolic which is a consequence of the following result.
	
	\begin{thm*}[\cite{DPU,BDP}]
		Let $\Lambda_f(U)$ be a robustly transitive set\footnote{A compact set $\Lambda$ is a robustly transitive set for $f$ if it is the maximal $f$-invariant set in some neighborhood $U$ and if, for every $g$ $C^1$-close to $f$, the maximal $g$-invariant set  $\Lambda_g(U)=\bigcap_{n \in \mathbb{Z}}g^n(U)$ is also compact and $g:\Lambda_g \to \Lambda_g$ is transitive.} and $E_1\oplus \cdots \oplus E_k$, $E_1\prec E_2 \prec \cdots \prec E_k$, be its finest dominated splitting. Then $\Lambda_f(U)$ is a volume hyperbolic set, that is, there exists an integer $\ell \geq 1$ such that $Df^{\ell}$ uniformly contracts  the volume in $E_1$ and uniformly expands the volume in $E_k$. 
	\end{thm*}
\end{itemize}

\subsubsection{Key obstructions vs. dominated splitting for non-invertible endomorphisms} 
For the case of non-invertible endomorphisms, the situation changes dramatically. On the one hand, the existence of a source no longer is an obstruction, we can consider, for example, an expanding map. On the other hand, as it was said before, there are examples of local diffeomorphisms on surfaces without dominated splitting. Moreover, when it is considered endomorphisms having critical points, the full-kernel obstruction (which was first introduced in \cite{CW1}) plays an essential role.

\begin{keylemma}\label{obstruction:rt}
	There are no robustly transitive endomorphisms exhibiting full-dimensional kernel.
\end{keylemma}

In order to prove this, we use the following classical tool in $C^1$-perturbative arguments introduced by John Franks in \cite{Franks} for diffeomorphisms that can be easily adapted for endomorphisms as follows.

\begin{Franks' lemma}\label{lemma:Franks}
	Given $\mathcal{U}$ open set in $\mathrm{End}^1(M)$ and $f \in \mathcal{U}$, there exist $\varepsilon >0$ such that for every finite collection of distinct points $\Sigma=\{x_0,...,x_n\}$ in $M$, and linear maps
	$$L_i:T_{x_i}M \to T_{f(x_i)}M \ \ \text{such that} \ \ \|L_i-Df_{x_i}\|<\varepsilon, \ \ \text{for} \ \ 0\leq i \leq n,$$
	there exist an endomorphism $\hat{f} \in \mathcal{U}$, a neighborhood $B$ of $\Sigma$, and a family of balls $\{B_i\}_{i=0}^n$ contained in $B$, where $B_i$ is centered at $x_i$, verifying:
	\begin{align*}
	&\hat{f}(x_i)=f(x_i) \ \ \text{and} \ \ \hat{f}\mid_{B_i}=L_i, \ \ \text{for each} \ \ 0\leq i \leq n; \\& \text{and} \ \ \hat{f}(x)=f(x), \ \ \text{for every} \ \ x \in M\backslash B,
	\end{align*}
	where by abuse of notation $\hat{f}\mid_{B_i}=L_i$  means the action of $\hat{f}$ in each $B_i$ is equal to the linear map $L_i$. 
\end{Franks' lemma}

Thus, we can conclude that the \hyperref[obstruction:rt]{Key Obstruction Lemma} follows from \hyperref[lemma:Franks]{Franks' Lemma} applied to $\Sigma=\{x,f(x),\dots,f^{m-1}(x)\}$ and $L_i=Df_{f^i(x)}$ for each $i=0,1,\dots,m-1$, where $\ker(Df^m_x)$ is $d$-dimensional. Hence, there is an endomorphism $\hat{f}$ arbitrarily close to $f$ so that $\hat{f}$ is equal to $L_i=Df_{f^i(x)}$ around $f^i(x)$ for each $i=0,1,\dots,m-1$. In particular, $\hat{f}^m$ is equal to $Df^m_x$ around $x$ and, hence, the image of such neighborhood of $x$ by $\hat{f}^m$ is exactly one point which contradicts transitivity. Therefore, $f$ cannot be a robustly transitive endomorphism.

\subsubsection{Two-dimensional endormorphisms with critical points}

Let us quickly comment about the approach in \cite{CW1}. For a robustly transitive surface endomorphism $f$ displaying critical points having nonempty interior, one can define the set $\Lambda$ consisting of all the orbits $(x_i)_i$ which get into the interior of the critical set infinitely many times for the past and the future. That is, $(x_i)_i \in \Lambda$ if and only if $x_i \in \mathrm{int}(\mathrm{Cr}(f))$ for infinitely many $i<0$ and infinitely many $i>0$. In order to simplify the notation, let us also denote by $f$ the map on the space of all the orbits defined by $(x_i)_i \mapsto (f(x_i))_i=(x_{i+1})_i$. Then, one has that $\Lambda$ is $f$-invariant and, moreover, one can define for every $(x_i)_i$ in $\Lambda$ an invariant splitting of the tangent bundle over $\Lambda$ as follows,
\begin{align}\label{eq:spliting}
E(x_i)=\ker(Df_{x_i}^{\tau_i^{+}+1}) \ \ \text{and} \ \ F(x_i)=\mathrm{Im}(Df_{x_{i+\tau_i^{-}}}^{|\tau_i^{-}|}),
\end{align}
where $\tau_i^{+}\geq 0$ is the time that $x_i$ takes to enter in the critical set for the first time, and $\tau_i^{-} < 0$ is the time that $x_{i}$ takes to go back to the critical set along the orbit $(x_i)_i$ for the first time. In particular, we have $x_{i+\tau_i^{+}}$ and $x_{i+\tau_i^{-}}$ in $\mathrm{Cr}(f)$.

The splitting $E\oplus F$ is well defined over $\Lambda$ because the full-dimensional kernel obstruction guarantees that $\ker(Df^n_{x_i})$ is at most one-dimensional for any $n\geq 0$. In particular, $E$ is one-dimensional and $Df$-invariant. Furthermore, it is used the definition of $\tau_i^{-}$ together with the fact that $\dim\ker(Df^n_{x_i}) \leq 1$, for all $n\geq 1$, to show that $F$ is $Df$-invariant. Then, to get that $E\oplus F$ is a dominated splitting, it is proved  in \cite{CW1} that the absence of domination property allows to create a point having full-dimensional kernel for some $C^1$-perturbation of $f$ which is incompatible with robust transitivity. Thus, $E\oplus F$ is a dominated splitting and it can be extended to the closure of $\Lambda$ which is the whole space. Finally, we use that every robustly transitive endomorphism displaying critical points is approximated by such kind endomorphisms having dominated splitting, which allows us to push the domination splitting to the limit to conclude Theorem~\ref{thm A} for surface endomorphisms.

In higher dimension, the kernel of $Df^n$ may have distinct dimensions depending on $n$, even if the kernel of $Df$ has constant dimension. Moreover, the subbundles $E$ and $F$ on $\Lambda$ may have not constant dimension along the orbit on $\Lambda$.

In the sequel, we explain our strategy to figure out that obstacle and then proving Theorem~\ref{thm A}.

\subsection{Sketch of the proof of Theorem~\ref{thm A}}\label{subsec:sketch}
Let $f_0$ be a robustly transitive endomorphism displaying critical points and $\mathcal{U}_0$ a neighborhood of $f_0$ in $\mathrm{End}^1(M)$ such that every endomorphism in it is (robustly) transitive. Up to shrink $\mathcal{U}_0$, we can find $1\leq \kappa <d$ as the smallest integer satisfying
\begin{align}\label{eq:kappa}
\dim \ker(Df^m) \leq \kappa, \ \ \forall f \in \mathcal{U}_0,\ \ m\geq 1,
\end{align}
where $\dim \ker(Df)=\max_{x \in M}\dim \ker(Df_x)$. Otherwise, we could find $f \in \mathcal{U}_0$ and $m\geq 1$ such that $\dim\ker(Df^m)=d$, which by \hyperref[obstruction:rt]{Key Obstruction Lemma} is absurd since $f$ is also a robustly transitive endomorphism.

Since $\kappa$ is chosen as the smallest integer satisfying \eqref{eq:kappa}, it follows that $f_0$ can be approximated by $f \in \mathcal{U}_0$ satisfying the equality for some $m\geq 1$.
Let us define $m_f$ as the smallest positive integer $m$ such that
\begin{itemize}
	\item $\{x \in M:\dim \ker(Df^m_x)=\kappa\}$ has nonempty interior; or
	\item if such subset above has empty interior then we take $m_f$ as the smallest one such that $\dim \ker(Df^m_x)=\kappa$, for some $x \in M.$
\end{itemize}

In order to avoid any confusion, we point out that the second item must be considered if and only if the interior of $\{x \in M:\dim \ker(Df^m_x)=\kappa\}$ is empty.

From now on, let us define $\mathrm{Cr}_{\kappa}(f)$ as the set $\{x \in M:\dim \ker(Df^{m_f}_x)=\kappa\}$. This set plays an important role in our approach. Furthermore, we denote the set of all endomorphisms $f$ in $\mathcal{U}_0$ which $\mathrm{Cr}_{\kappa}(f)$ has nonempty interior by $\mathcal{F}_0$.

It should be noted that $\mathcal{F}_0$ is nonempty and accumulates at $f_0$. Indeed, given $x \in \mathrm{Cr}_{\kappa}(f)$ for some $f$ close to $f_0$,
we can apply \hyperref[lemma:Franks]{Franks' Lemma} to $\Sigma=\{x,f(x),\dots,f^{m_f-1}(x)\}$ and 
$L_i=Df_{f^i(x)}$, for each $i=0,1,\dots,m_f-1$, to get an endomorphism $\hat{f}$ $C^1$-close to $f$ and, therefore, close to $f_0$, such that $ \ker(D\hat{f}^{m_f}_y)$ is $\kappa$-dimensional for every $y$ near $x$. Then, we conclude that $\mathrm{Cr}_{\kappa}(\hat{f})$ has nonempty interior and, moreover, $m_{\hat{f}}\leq m_f$.

We now define the set $\Lambda$ and the splitting $E\oplus F$, in our context of higher dimension.
For $f \in \mathcal{F}_0$, we define
\begin{align}\label{def:lambda}\tag{$\ast$}
\Lambda_f=\left\{(x_i)_i \subseteq M\left|\begin{array}{l}
x_{i+1}=f(x_i), \, \forall i \in \mathbb{Z}, \ \ \text{and} \ \ \exists (i_n)_{n} \subseteq \mathbb{Z} \ \ \text{such that} \\ x_{i_n} \in \mathrm{Cr}_{\kappa}(f) \ \  \text{for infinitely many} \ \ i_n<0 \ \ \text{and} \ \ i_n>0
\end{array}\right.\right\}
\end{align}
and
\begin{align}\label{def:EF}\tag{$\ast\ast$}
E(x_i)=\ker(Df_{x_i}^{m_f+\tau^{+}_i}) \ \ \text{and} \ \ F(x_i)=\mathrm{Im}(Df_{x_{i+\tau^{-}_i}}^{|\tau^{-}_i|}),
\end{align}
where 
\begin{align*}
\begin{array}{c}
\tau^{+}_i=\min\{n\geq 0:x_{i+n} \in \mathrm{Cr}_{\kappa}(f)\}, \ \  \text{and} \ \
\tau^{-}_i=\max\{n\leq -m_f:x_{i+n} \in \mathrm{Cr}_{\kappa}(f)\}.
\end{array}
\end{align*}

Note that $\tau^{-}_i$ and $\tau^{+}_i$ are slightly different from the ones in \cite{CW1} (recall definition in \eqref{eq:spliting}). However, it should be noted that $m_f$ is the time that $\ker(Df^n)$ have maximal dimension in $\mathcal{U}_0$ and $\mathrm{Cr}_{\kappa}(f)$ is the set such that the kernel of $Df^{m_f}$ has maximal dimension. In particular, if $f$ is a surface endomorphism, we have that $m_f=1, \kappa=1$, $\mathrm{Cr}_1(f)$ is the critical set of $f$, and $\tau^{\pm}_i$ are the same as in \cite{CW1}.

We remark that (up to taking a subset) $\mathcal{F}_0$ is the natural candidate to prove Theorem~\ref{thm A1}. Then, it remains to show that there is an integer $\ell>0$ and a number $\alpha>0$ such that $E\oplus F$ is an $(\alpha,\ell)$-dominated splitting. To prove that, we will see in Lemma~\ref{lambdaset-non-empty} that $\Lambda_f$ given by \eqref{def:lambda} is dense in the inverse limit space (see \S\ref{sec:weak-hyp}) and, in Proposition~\ref{cont-ext},  $E\oplus F$ can be extended to the closure of $\Lambda_f$ once provided that $E\oplus F$ is an $(\alpha,\ell)$-dominated splitting.

To prove such uniform behavior, we state a technical dichotomy as follows. However, before doing it,  we would like to emphasize that to define $E$ and $F$ in \eqref{def:EF} it is  used only the fact that $\Lambda_f\neq \emptyset$ and $\dim \ker(Df^m)\leq \kappa$ for all $m\geq 1$.

\begin{thm}\label{thm B}
	Let $f_0$ be an endomorphism displaying critical points. Assume that there is an integer $1\leq \kappa<d$ and a set $\mathcal{F}$ consisting of endomorphisms converging to $f_0$ such that every $f \in \mathcal{F}$ satisfies that $\Lambda_f\neq \emptyset$ and
	$\dim \ker(Df^m)\leq \kappa, \forall m \in \mathbb{Z}$. 
	Then, only one of the following statement holds:
	\begin{itemize}
		\item either there exist $\ell>0$ and $\alpha >0$ such that for each $f \in \mathcal{F}$, the splitting $E\oplus F$ is ($\alpha,\ell$)-dominated splitting over $\Lambda_f$;
		\item or $f_0$ is accumulated by endomorphisms $g$ which $\ker(Dg^m)$ has dimension greater than $\kappa$, for some $m \geq 1$.
	\end{itemize}
\end{thm}

We have already shown that $\mathcal{F}_0$ satisfies the hypothesis of Theorem~\ref{thm B} and $f_0$ cannot be approximated by endomorphisms whose kernel has dimension greater than $\kappa$. Then, Theorem~\ref{thm A1} follows from Theorem~\ref{thm B} since Lemma~\ref{lambdaset-non-empty} implies  the density of $\Lambda_f$, and in Proposition~\ref{cont-ext} we prove that the dominated splitting can be extended to the closure of $\Lambda_f$ which is the space of all the orbits.

\subsubsection{Novelties and new techniques} We want to emphasize the new approaches brought by the present paper that differ with the ones developed for diffeomorphisms and surface endomorphisms having critical points.

\begin{itemize}
	\item It is used the kernel of $Df$ at the critical set (that could be multidimensional and have different dimension at distinct points) to build a candidate for a dominated splitting on a dense set. On the one side, this is substantially different on how the splitting is built for the case of diffeomorphisms where it is used the splitting over the periodic points. On the other side, the strategy goes beyond the approach for surface endomorphisms where the kernel of $Df^n$ has dimension one for any point in the critical set and any iterated $n$.
	
	\item A dominated splitting defined on an invariant non-compact set could not be extended to the closure (see Example~\ref{ex}). Therefore, a fine control on the angle between multidimensional subbundles have to be brought into consideration, an issue that is not present in previous approaches.
\end{itemize}

\subsection{How the paper is organized}
In \S\ref{sec:weak-hyp}, we discuss the notion of dominated splitting and some related properties, the equivalence of dominated splitting via cone criterion. In \S\ref{sec:proof of thm A}, we use Theorem~\ref{thm A1} to prove Theorem~\ref{thm A}. Finally, \S\ref{sec:proof of thm B} is devoted to the proof of Theorem~\ref{thm B} and recall that Theorem~\ref{thm A1} follows from Theorem~\ref{thm B}.

\section{Weak form of hyperbolicity for endomorphisms}\label{sec:weak-hyp}

In this section will be formalized the notion of dominated splitting in terms of invariant splitting for endomorphisms displaying critical points. Further, we will present some fundamental properties that will be useful throughout this paper.

Due to the fact that for an endomorphism a point may exhibit more than one preimage, it is natural to consider the \textit{inverse limit space of $M$ with respect to $f$},
\begin{align}\label{inverse-limit-space}
M_f=\{(x_i)_i: x_i \in M \ \ \text{and} \ \  f(x_{i-1})=x_i, \forall i \in \mathbb{Z}\}.
\end{align}

It is a compact metric space and the natural projection $(x_i)_i \mapsto x_0$ is continuous. Moreover, an endomorphism $f \in \mathrm{End}^1(M)$ induces a homeomorphism on $M_f$ defined by $(x_i)_i \mapsto (f(x_i))_i=(x_{i+1})_i$, whenever there is no confusion, it will also be denoted by $f$. The points in $M_f$ are called the orbits of $f$. A subset $\Lambda$ of $M_f$ is said to be \textit{f-invariant} or, simply, \textit{invariant} if $f^{-1}(\Lambda)=\Lambda$.

The concept of dominated splitting is established in the context of endomorphisms without critical points (i.e., local diffeomorphisms and diffeomorphisms) and its definition is the following.

We say that an $f$-invariant set $\Lambda \subseteq M_f$ admits a \textit{dominated splitting of index $\kappa$} for $f$ if for all $(x_i)_i \in \Lambda$ there exist two families $(E(x_i))_i$ and $(F(x_i))_i$ of $\kappa$ and $(d-\kappa)$-dimensional subspaces satisfying:
\begin{description}
	\item[Invariant splitting:] for each $i \in \mathbb{Z}$, one has that
	$$\left.\begin{array}{ll}
	Df(E(x_i))=E(f(x_i)), \ \ Df(F(x_i))=F(f(x_i)),\\
	\text{and} \ \ T_{x_i}M=E(x_i)\oplus F(x_i);
	\end{array}\right.$$
	
	\item[Domination property:] there exists $\ell \geq 1$ such that for each $i \in \mathbb{Z}$ and unit vectors $u \in E(x_i)$ and $v \in F(x_i)$, one has that
	$$\|Df^{\ell}(u)\|\leq \frac{1}{2}\|Df^{\ell}(v)\|.$$
\end{description}

When $f$ is a diffeomorphism, it is used $M$ instead of $M_f$ in the definition above. Recall that the domination property is denoted by $E \prec F$ or $E\prec_{\ell} F$ if we want to emphasize the role of $\ell$.

\begin{remark}\label{rem.subbundles}
	The families $E$ and $F$ are, actually, subbundles of the vector bundle of $M_f$ defined by $TM_f=\{((x_i)_i,v) : v \in T_{x_0}M\}$. Moreover, the bundle $E$ depends only on the forward orbit while $F$ depends only on the backward orbit of a point in $M_f$. This implies that the bundle $E$ induces a subbundle of $TM$ and this will be used in the proof of Corollary~\ref{cor:topobst}. 
\end{remark}

\begin{remark}\label{rmk:norm-conorm}
	Let $W$ be any inner product space and $\Phi:W \to W$ be a linear map. For a subspace $V$ of $W$, we denote $\Phi$ restricted to $V$ by $\Phi\mid_V$ and, respectively, define the norm and conorm of $\Phi\mid_V$ by:
	\begin{align}
	\|\Phi\mid_{V}\|=\max_{v \in V\backslash \{0\}}\frac{\|\Phi(v)\|}{\|v\|} \ \ \text{and} \ \ \m(\Phi\mid_{V})=\min_{v\in V\backslash\{0\}}\frac{\|\Phi(v)\|}{\|v\|}.
	\end{align}
	When $V$ is the whole space, we simply say that $\|\Phi\|$ and $\m(\Phi)$ are the norm and conorm of $\Phi$.
\end{remark}

From now on, by Remark~\ref{rmk:norm-conorm}, we rewrite the inequality in the domination property:
$$\|Df^{\ell}\mid_{E(x_i)}\|\leq \frac{1}{2}\m(Df^{\ell}\mid_{F(x_i)}).$$

It should be pointed out some differences when the derivative is not invertible everywhere.  For instance, suppose that $E \oplus F$ is a splitting over an $f$-invariant set $\Lambda \subseteq M_f$ verifying all the properties of the definition above. Note that if $u_E+u_F \in \ker(Df_{x_i})$, where $u_E \in E(x_i)$ and $u_F \in F(x_i)$, then $Df(u_E)$ is parallel to $Df(u_F)$ which affects the invariance property. Moreover, if $u_F \in  \ker(Df_{x_i})$ then the domination property implies that $ E(x_i)$ is contained in $\ker(Df_{x_i})$, and so, neither $E(x_i)$ nor $F(x_i)$ are invariant. Therefore, in order to extend the notion of dominated splitting, we require the following:
$$\ker(Df_{x_i}^n) \subseteq E(x_i), \ \ \textrm{for each} \ \ x_i \ \ \text{in} \ \ (x_i)_i \in \Lambda \ \ \text{and each} \ \ n\geq 1.$$

In addition, we also require that the angle between $E$ and $F$ is uniformly away from zero. This will allow to extend the dominated splitting to the closure of $\Lambda$, see Proposition~\ref{cont-ext}. Otherwise, it may exist an orbit $(x_i)_i$ such that $E(x_i)=\ker(Df_{x_i})$ and the angle of $E(x_i)$ and $F(x_i)$ goes to zero as $i$ goes to $+\infty$, then on somewhere on the boundary of $\Lambda$ the extension of $E$ and $F$ have some intersection. See the following example.

\begin{example}\label{ex}
	Let $(A_n)_n$ be a sequence of square matrices defined by $$A_n=\left(\begin{array}{cc}
	0 & 1\\
	0 & (n+1)^{-1}
	\end{array}\right), \, \forall n \geq 1.$$
	
	Take $E:=E_n$ and $F_n$ as the subspaces generated by $v=(1,0)$ and $v_n=(1,1/n)$, respectively. Then, $E_n \oplus F_n$ is a dominated splitting for $(A_n)_n$ since $A_n(E_n) \subseteq E_{n+1}$ and $A_n(F_n)=F_{n+1}$. Moreover, $$\|A_n\mid_{E_n}\|=0\leq \dfrac{1}{2}\|A_n\mid_{F_n}\|,$$ for each $n \geq 1$. However, $E_n,$ and $F_n$ converge to the same subspace $E$. 
	
	This example shows the existence of a dominated splitting along an orbit which cannot be extended to the closure.
\end{example}

Thus, before proposing the definition of dominated splitting for endomorphism displaying critical points, we should introduce the notion about angle between subspaces.

The $\textit{angle}$ between the non-zero vectors $v,w \in T_xM$ with respect to the metric $\langle \cdot,\cdot\rangle$ (which, for simplicity, we ignore the dependence of the inner product on $x$ in $M$) is defined as the unique number $\varangle(v,w)$ in $[0,\pi]$ satisfying $\cos\varangle(v,w)=\langle v,w\rangle /\|v\|\|w\|$. Then, given $V$ and $W$ two nontrivial subspaces of $T_xM$, we define the angle between them by:
\begin{align}
\varangle(V,W)=\min_{v \in V\backslash\{0\}} \min_{w \in W\backslash\{0\}} \varangle(v,w).
\end{align}

The angle between two subspace is a number contained in $[0,\pi/2]$. We also write $\varangle(\mathbb{R}v,W)$ to refer the angle between the space generated by a non-zero vector $v$ and the subspace $W$.

We would like to emphasize that $\varangle(V,W)=0$ does not mean that $V=W$, it just means that the intersection between $V$ and $W$ is nontrivial; and $\varangle(V,W)>\alpha$ for some $\alpha>0$ means that $V$ and $W$ are far away from each other. This notion of angle between two subspaces will be useful to extend the definition of dominated splitting in the context of endomorphisms displaying critical points. Furthermore, that angle allows us to define a distance on $\mathrm{Grass}_r(T_xM)$, called \textit{$r$-dimensional Grassmannian of $T_xM$}, which consists of all the $r$-dimensional subspaces of $T_xM$. For all $V$ and $W$ in $\mathrm{Grass}_r(T_xM)$, we define the distance between them by:
\begin{align}\label{dist}
dist(V,W)=\cos \varangle(V^{\perp},W),
\end{align}
where $V^{\perp}$ denotes the orthogonal complement of $V$. It is well known that the Grassmannian endowed with this distance is a compact metric space. For more details about the distance see \cite[Appendix A.1]{BPS}

Now, we are able to define a dominated splitting for endomorphisms exhibiting critical points.

\begin{definition}\label{defi:DS} Let $f \in \mathrm{End}^1(M)$ be an endomorphism displaying critical points.	An invariant subset $\Lambda$ of $M_f$ admits a dominated splitting of index $\kappa$ for $f$ if for all $(x_i)_i \in \Lambda$ there exist two families $(E(x_i))_i$ and $(F(x_i))_i$ of $\kappa$ and $(d-\kappa)$-dimensional subspaces such that the following properties hold. 
	\begin{description}
		\item[Invariant splitting:] for each $i \in \mathbb{Z}$, one has that
		$$\left.\begin{array}{ll}
		Df(E(x_i))\subseteq E(f(x_i)), \ \ Df(F(x_i))=F(f(x_i)),\\
		\text{and} \ \ T_{x_i}M=E(x_i)\oplus F(x_i);
		\end{array}\right.$$
		
		\item[Uniform angle:] there exists $\alpha >0$ such that $\varangle (E(x_i),F(x_i))\geq \alpha$ for each $i \in \mathbb{Z}$;
		
		\smallskip
		
		\item[Domination property:] there exists $\ell \geq 1$ such that $E\prec_{\ell} F$.
	\end{description}
\end{definition}

We will say that $E\oplus F$ is an $(\alpha,\ell)$-dominated splitting if we want to emphasize the role of $\ell$ and $\alpha$. When $\Lambda=M_f$, we also say that \textit{$f$ has a dominated splitting}. For simplicity, we will write $E\prec F$ instead of $E\prec_{\ell} F$ when there is no confusion.

\begin{remark}\label{rmks}
	Observe that if $E$ and $F$ satisfy the items in Definition~\ref{defi:DS}, then $Df\mid_F$ is an isomorphism and $\ker(Df_{x_i}^n) \subseteq E(x_i)$ for each $(x_i)_i \in \Lambda$ and $n\geq 1$.
\end{remark}

\begin{remark}
	It is not required any uniform angle property to define dominated splitting for endomorphisms displaying critical points in the introduction. The uniformity of the angle follows from the fact that the  manifold is compact and the subbundles are continuous.
\end{remark}

More general, an invariant subset $\Lambda$ of $M_f$ admits a dominated splitting of index $\kappa$ for $f$ if for all $(x_i)_i \in \Lambda$ there exist non-trivial families $(E(x_i))_i$ and $(F_j(x_i))_i, \ \ 1\leq j \leq r$,   of $\kappa$ and $d_j$-dimensional subspaces with $d_1+d_2+\cdots+d_r=d-\kappa$ satisfying:
\begin{description}
	\item[Invariant splitting:] for each $i \in \mathbb{Z}$, one has that
	$$Df(E(x_i))\subseteq E(f(x_i)), \ \ Df(F_j(x_i))=F_j(f(x_i)), \ \
	\text{and} \ \ T_{x_i}M=E(x_i)\oplus_{j=1}^r F_j(x_i);
	$$
	\item[Uniform angle:] there exists $\alpha >0$ so that $\varangle (E(x_i),\oplus_{j=1}^rF_j(x_i))\geq \alpha$ for each $i \in \mathbb{Z}$;
	\smallskip
	
	\item[Domination property:] $E\prec F_1 \prec F_2 \prec \cdots \prec F_r$.
\end{description}

\subsection{Dominated splitting properties}\label{subsec:ds property}
Throughout this section, we adapt to the context of endomorphisms displaying critical points some of the main properties about dominated splitting which appear in \cite{Crovisier-Potrie} in the diffeomorphisms context.

Let $f$ be an endomorphism displaying critical points and $E\oplus F$ be a dominated splitting over $f$-invariant subset $\Lambda$ of $M_f$. 

The uniqueness of the dominated splitting is guaranteed by the following.

\begin{prop}\label{uniqueness}
	If $G\oplus H$ is a dominated splitting over $\Lambda$ for $f$, which holds $\dim E=\dim G$, then $E(x_i)=G(x_i)$ and $F(x_i)=H(x_i)$ for all $x_i$ in the orbit $(x_i)_i \in \Lambda$.
\end{prop}

The proof needs some preliminaries.

\begin{lemma}\label{lemma:uniqueness}
	If $G\oplus H$ is a dominated splitting over $\Lambda$ for $f$ with $\dim E(x_i)\leq \dim G(x_i)$, for all $(x_i)_i \in \Lambda$, then $E(x_i)\subseteq G(x_i)$. In particular, $E\oplus (G\cap F) \oplus H$ is a dominated splitting over $\Lambda$ for $f$.
\end{lemma}

\begin{proof}
	We can, without loss of generality, choose $\ell \geq 1$ such that $E\prec_{\ell} F$ and $G\prec_{\ell} H$. Moreover, it follows from Remark~\ref{rmks} that $Df\mid_{H}$ is an isomorphism and $\ker(Df^m_{x_i})$ must be contained in $E(x_i)$ and $G(x_i)$, for each $(x_i)_i \in \Lambda$ and $m\geq 1$.
	
	To conclude the proof of the proposition it remains to show that if $u \in E(x_i)$ such that $Df^m(u)\neq 0$, for every $m\geq 1$, then $u \in G(x_i)$. Then, for every $u \in E(x_i)$, one can decompose $u=u_G+u_H$ in a unique way where $u_G \in G(x_i)$ and $u_H \in H(x_i)$. Analogously, one can decompose $u_H=u'_E+u'_F$ where $u'_E \in E(x_i)$ and $u'_F \in F(x_i)$. Then $u'_F$ must be zero. Otherwise, we get that
	\begin{align*}
	\|Df^{k\ell}(u)\|&\geq \|Df^{k\ell}(u_H)\|-\|Df^{k\ell}(u_G)\|\geq \left(1-\frac{1}{2^{k}}\right)\|Df^{k\ell}(u_H)\|\\
	&\geq \left(1-\frac{1}{2^{k}}\right) \left(\|Df^{k\ell}(u'_F)\|-\|Df^{k\ell}(u'_E)\| \right)\\
	&\geq \left(1-\frac{1}{2^{k}}\right)^2 \|Df^{k\ell}(u'_F)\|,
	\end{align*}
	implying that $\|Df^{k\ell}(u)\|$ and $\|Df^{k\ell}(u'_F)\|$ have the same growth which is impossible.
	Therefore, $u_H \in E(x_i)\cap H(x_i)$ and $u_G \in E(x_i)\cap G(x_i)$. Symmetrically, we deduce that if $v \in G(x_i)$ whose $v=v_E+v_F$ then $v_E \in G(x_i)\cap E(x_i)$ and $v_F \in G(x_i)\cap F(x_i)$. Moreover, since $\dim E(x_i)\leq \dim G(x_i)$, we have that either $G(x_i)=E(x_i)$ or $G(x_i)\cap F(x_i)\neq\{0\}$.
	
	Take non-zero vectors $u \in E(x_i)\cap H(x_i)$ and $v \in G(x_i)\cap F(x_i)$. Then, as $G \prec H$, we deduce that $\|Df^{\ell}(u)\|$ grows faster than $\|Df^{\ell}(v)\|$ which contradicts the fact that $E \prec F$. Thus, at least one of these intersection $E(x_i)\cap H(x_i)$ and $G(x_i)\cap F(x_i)$ is trivial. Thus, we obtain that $E(x_i)\cap H(x_i)=\{0\}$. It implies that $E(x_i)$ is contained $G(x_i)$. To conclude the proof, we should observe that $G(x_i)\cap F(x_i)$ is invariant and $E(x_i)\prec G(x_i)\cap F(x_i) \prec H(x_i)$.
\end{proof}

As directly consequence of the previous lemma, we conclude the first part of the uniqueness of the dominated splitting.

\begin{corollary}\label{cor:uniqueness}
	If $G\oplus H$ is a dominated splitting over $\Lambda$ for $f$ such that $\dim E=\dim G$ then $E(x_i)=G(x_i)$ for all $x_i$ in the orbit $(x_i)_i \in \Lambda$.
\end{corollary}

To complete the proof of Proposition~\ref{uniqueness}, we state the following.

\begin{lemma}\label{F:uniqueness}
	If $E\oplus F$ and $E\oplus H$ are dominated splittings over $\Lambda$ for $f$ then $F(x_i)=H(x_i)$ for all $x_i$ in the orbit $(x_i)_i$ in $\Lambda$.
\end{lemma}

\begin{proof}
	First, we assume that $(x_i)_i \in \Lambda$ with $x_i \notin \mathrm{Cr}(f)$ for all $i \in \mathbb{Z}$. Then, since $Df_{x_i}:T_{x_i}M \to T_{x_{i+1}}M$ is an isomorphism, one has that for every unit vectors $u \in E(x_i)$ and $v \in F(x_i)$ 
	\begin{align*}
	\|Df^{-\ell}(u)\|\geq \frac{1}{\|Df^{\ell}\mid_{E}\|}\geq 2\frac{1}{\m(Df^{\ell}\mid_{F})}\geq 2\|Df^{-\ell}(v)\|.
	\end{align*}
	In particular, $F\prec_{\ell} E$ for $Df^{-1}$ along of the orbit $(x_i)_i$ and, similarly, $H\prec_{\ell} E$ for $Df^{-1}$.
	Therefore, applying Corollary~\ref{cor:uniqueness}, one concludes that $F(x_i)=H(x_i)$ since $\dim F=\dim H$.
	
	In the general case, for a non-zero vector $u \in F(x_i)$ one deduces that there are two vectors $w \in E(x_i)$ and $v \in H(x_i)$, with $v\neq 0$, such that $u=w+v$. Assume without loss of generality that $i=0$. Then, using that $Df\mid_F$ and $Df\mid_H$ are isomorphisms, we can find the sequences $(u_j)_{j\leq 0}\subseteq F\backslash\{0\}, \, (w_i)_{j\leq 0} \subseteq E$, and $(v_j)_{j\leq 0} \subseteq H\backslash\{0\}$ such that $u_j=w_j+v_j$ and $Df(u_{j-1})=u_j$. Note that if $Df^m(w_{j-1})=0$, for some $1\leq m\leq |j|$, there is nothing to be proved since it direct implies that $u \in H(x_j)$. Hence, we can assume that $Df(\ast_{i-1})=\ast_i$ where $\ast=u,v,\, \text{and}\ \ w$ and $\ast_0=\ast$. 
	
	Denote $V_i=\mathrm{span}\{w_i,v_i\}$ the span of $w_i,v_i$; and define  $A_i:V_{i-1} \to V_i$ as the restriction of $Df_{x_{i-1}}$ to $V_{i-1}$, for each $i\leq 0$. Observe that defining $A_i^n:V_i \to V_{i+n}$ by $A_i^n=A_{i+n}^{-1}\cdots A_i^{- 1}$ for every $n \leq 0$, we can repeat the previous arguments in order to show that $V_i$ admits two dominated splittings for $A_i^{-1}$, which are $V_i=\mathrm{span}\{u_i\}\oplus \mathrm{span}\{w_i\}$ with $\mathrm{span}\{u_i\}\prec_{\ell} \mathrm{span}\{w_i\}$, and
	$V_i=\mathrm{span}\{v_i\}\oplus \mathrm{span}\{w_i\}$ with $\mathrm{span}\{v_i\}\prec_{\ell} \mathrm{span}\{w_i\}$ since $u_i \in V_i$ and the following holds
	$$\|A_i\mid_{\mathrm{span}\{w_i\}}\|\leq \|Df\mid_{E}\|,  \ \ \|A_i\mid_{\mathrm{span}\{u_i\}}\|\geq \m(Df\mid_{F}), \ \ \text{and} \ \ \|A_i\mid_{\mathrm{span}\{v_i\}}\|\geq \m(Df\mid_{H}).$$
	
	Therefore, we repeat the argument in proof of Lemma~\ref{lemma:uniqueness} to get that $v_0$ and $u_0$ are parallel. This completes the proof.
\end{proof}

We can use the uniqueness of the dominated splitting to show the continuity and extend it to the closure.

\begin{prop}\label{cont-ext}
	The subbundles $E$ and $F$ depend continuously\footnote{Continuity in this means that when considering local coordinates so that the tangent bundle becomes trivial, the bundles depend on the point continuously as subspaces of $\mathbb{R}^d$.} with the point $(x_i)_i \in \Lambda$ and the closure $\overline{\Lambda}$ in $M_f$ admits a dominated splitting which coincides with $E\oplus F$ in $\Lambda$. 
\end{prop}

\begin{proof}
	The proof is a standard argument once we have shown the uniqueness of the dominated splitting and the uniformity of the angle. This can be found in \cite{Crovisier-Potrie}.
\end{proof}

\begin{remark}\label{proj-E}
	It is well-known and follows from the proof of uniqueness that if  $E\oplus F$ is a dominated splitting, then the subbundle $E$ only depends of the forward orbits. That is, for all points $(x_i)_i$ and $(y_i)_i$ in $\Lambda$, one has that 
	\begin{align}\label{eq-E}
	E(x_i)=E(y_i), \forall i \geq 0, \ \ \text{whenever} \ \ x_0=y_0.
	\end{align}
	In particular, by Propositions \ref{uniqueness} and \ref{cont-ext} together with \eqref{eq-E}, the subbundle $E$ of $TM_f$ induces an invariant continuous subbundle of $TM$ which will be also denoted by $E$.
\end{remark}


\subsection{Cone-criterion}\label{subsec:cone-criterion}

In this section, we will show that the existence of dominated splitting for endomorphisms displaying critical points can also be characterized in terms of cone fields as well as in the diffeomorphisms context.

A \textit{cone-field $\mathscr{C}$ of dimension $k$ on $M$} is a continuous family of convex closed non-vanishing cones $\mathscr{C}(x)$ in $T_xM$ such that the subspaces of maximal dimension inside of $\mathscr{C}(x)$ are $k$-dimensional. The closure of the complement of $\mathscr{C}(x)$ in $T_xM$ is a cone of dimension $d-k$, called dual cone at $x$ and denoted by $\mathscr{C}^{\ast}(x)$. We  will be called \textit{dual cone-field of $\mathscr{C}$} the family $\mathscr{C}^{\ast}=\{\mathscr{C}^{\ast}(x):x \in M\}$. We say that a cone-field $\mathscr{C}$ is invariant by an endomorphism $f$ if there is an integer $\ell>0$ such that
$Df^{\ell}(\mathscr{C}(x))$ is contained in $\mathrm{int}(\mathscr{C}(f^{\ell}(x)))\cup\{0\}$, where $\mathrm{int}(\mathscr{C}(x))$ denotes the interior of $\mathscr{C}(x)$ in $T_xM$. We also can say that $\mathscr{C}$ is $\ell$-invariant if we want to emphasize the role of $\ell$.

We next state an equivalent notion of dominated splitting for endomorphisms $f:M \to M$ exhibiting critical points.

\begin{prop}\label{prop:domcone}
	Let $E\oplus F$ be a dominated splitting for $f$ with $\kappa=\dim E$. Then there is an $(d-\kappa)$-dimensional invariant cone-field $\mathscr{C}$ such that $E(x)\cap \mathscr{C}(x)=\{0\}$ for each $x \in M$.
\end{prop}

In the proof, we will first define the cone-field and, then, will prove that it is invariant and transversal to the kernel.

By Remark~\ref{proj-E}, $E$ is an invariant continuous subbundle of $TM$ and, by Remark~\ref{rmks}, we have that $\ker(Df^n_x)$ is contained in $E(x)$ for all $x \in M$. Then, for every $x \in M$, we can define a \textit{cone-field} $\mathscr{C}_{E}$ on $M$ of dimension $\kappa$ and length $\eta >0$ by:
\begin{align}
\mathscr{C}_{E}(x,\eta)=\{(u_1,u_2) \in E(x)\oplus E(x)^{\perp}:\|u_2\|\leq \eta \|u_1\|\}.
\end{align}

Note that $u=(u_1,u_2) \in E(x)\oplus E^{\perp}(x)$ satisfies that
$$\tan \varangle(\mathbb{R}u,E^{\perp}(x))\leq \tan \varangle(u,u_2)=\dfrac{\|u_1\|}{\|u_2\|},$$
then, using that cosine is decreasing on $[0,\pi/2]$, the cone-field can be rewritten as:
\begin{align*}
\mathscr{C}_{E}(x,\eta)=\{u \in E(x)\oplus E(x)^{\perp}:\varangle(\mathbb{R}u,E^{\perp}(x))\geq \arctan \eta^{-1}\}.
\end{align*}

From now on, since $\arctan : (0,+\infty) \to (0,\pi/2)$ is a homeomorphism preserving the orientation, we will make an abuse of notation rewriting, for each $\eta \in (0,\pi/2)$,
\begin{align}
\mathscr{C}_{E}(x,\eta)=\{u \in E(x)\oplus E(x)^{\perp}:\varangle(\mathbb{R}u,E^{\perp}(x))\geq \pi/2-\eta\}.
\end{align}

Recall that the dual cone-field of $\mathscr{C}_{E}(x,\eta)$ is the closure of $T_xM\backslash \mathscr{C}_{E}(x,\eta)$ which is $(d-\kappa)$-dimensional cone-field and will be written as
\begin{align}
\mathscr{C}_{E}^{\ast}(x,\eta)=\{u \in E(x)\oplus E(x)^{\perp}: \varangle(\mathbb{R}u,E(x))\geq \eta\}.
\end{align}

It is clear that $E(x) \subseteq \mathscr{C}_E(x,\eta)$, for each $x \in M$ and $\eta \in (0,\pi/2)$. Thus, we can conclude that $\ker(Df_x^n)\cap \mathscr{C}_E^{\ast}(x,\eta)=\{0\}$ for each $x \in M$ and $n\geq 1$, since $\ker(Df_x^n) \subseteq E(x)$.

Therefore, to conclude the proof of Proposition~\ref{prop:domcone}, we state the following.

\begin{claim*}
	There is $\alpha>0$ such that the dual cone-field $\mathscr{C}_E^{\ast}=\{\mathscr{C}_E^{\ast}(x,\alpha):x\in M\}$ is invariant.
\end{claim*}

Indeed, since the angle between $E$ and $F$ is bounded away from zero, there is a number $\alpha>0$ small enough satisfying $\varangle(F(x_i),E(x_i))\geq 2\alpha$, for each $i \in\mathbb{Z}$ in all the orbit $(x_i)_i$. In particular, we have that the direction $F(x_0)$ associated to any orbit $(x_i)_i$ with $x_0=x$ is contained in $\mathscr{C}_E^{\ast}(x,\alpha)$. Moreover, there is a constant $c>0$ (depending only on $\alpha$) such that any decomposition of a vector $w \in \mathscr{C}_E^{\ast}(x,\alpha)$ as $w=w_E+w_F$ with $w_E \in E(x)$ and $w_F$ in direction $F$ at $x$ satisfies $\|w_F\|\geq c\|w_E\|$. Otherwise, $w$ cannot belong to $\mathscr{C}_E^{\ast}(x,\alpha)$. Thus, we obtain that
\begin{align*}
\|Df^{k\ell}(w_E)\|&\leq \|Df^{k\ell}\mid_{E}\|\|w_E\|\leq\left(\dfrac{1}{2}\right)^k\m(Df^{k\ell}\mid_{F})\|w_E\|\\
&\leq \left(\dfrac{1}{2}\right)^k \|Df^{k\ell}(w_F)\|\dfrac{\|w_E\|}{\|w_F\|}\leq
\left(\dfrac{1}{2}\right)^k c^{-1} \|Df^{k\ell}(w_F)\|
\end{align*}
and, hence, we can take $k>0$ large enough such that $\mathscr{C}_E^{\ast}(x,\alpha)$ is $k$-invariant for every $x \in M$. Take $\mathscr{C}$ as the dual cone-field $\mathscr{C}_E^{\ast}$, we conclude the proof.  

Conversely, we will show that the existence of an invariant cone-field transversal to the kernel is sufficient to get a dominated splitting. More precisely, we state the following.

\begin{prop}\label{prop:conedom}
	If $\mathscr{C}$ is an invariant $(d-\kappa)$-dimensional cone-field by $f$ such that $\ker(Df_x^n)\cap \mathscr{C}(x)=\{0\}$ for each $x \in M$ and $n\geq 1$, then there exist two nontrivial subbundles $E$ and $F$ on $M_f$ such that $E\oplus F$ is a dominated splitting for $f$, which $E$ is $\kappa$-dimensional.
\end{prop}

To prove Proposition~\ref{prop:conedom}, some preliminaries are needed. For a linear transformation $\Phi: V \to W$ where $V$ and $W$ are $d$-dimensional vector spaces with an inner product, we denote by $\Phi^{\ast}:W \to V$ its adjoint. It is well known that if the kernel of $\Phi$ is $r$-dimensional then there are $\{v_1,\dots,v_d\}$ and $\{w_1,\dots,w_d\}$ orthonormal bases of $V$ and $W$, respectively, such that
$\Phi(v_i)=0,\ \ \Phi^{\ast}(w_i)=0$ for each $1\leq i \leq r$; and $\Phi(v_i)=\sigma_i w_i, \ \ \Phi^{\ast}(w_i)=\sigma_i v_i$ with $\sigma_i>0$ for each $r<i\leq d$. Moreover, we have that $\sigma_i=\|\Phi(v_i)\|=\|\Phi^{\ast}(w_i)\|$ for each $i=1,\dots,d$.

We call $\sigma_i$'s the \textit{singular values of $\Phi$}. Note that the singular values of $\Phi$ are the eigenvalues of $\sqrt{\Phi^{\ast}\Phi}$ associated to the eigenvectors $v_i$'s, and we will order them as the following:
$$\sigma_1 = \cdots = \sigma_r=0 < \sigma_{r+1}\leq \cdots \leq \sigma_d.$$

When the kernel of $\Phi$ is trivial (i.e., $r=0$), we have $0<\sigma_1\leq \sigma_2 \leq \cdots \leq \sigma_d$.

Note that $\sigma_i$'s are also the singular values of $\Phi^{\ast}$. Moreover, the singular values are typically ordered from largest to smallest, but we are taking the opposite convention. For more information see \cite[Theorem 7.3.2]{MatrixAnalysis}.

We would also like to introduce a useful ``minimax'' characterization of the singular values, which appears in \cite[Appendix A]{BPS}:
\begin{align}\label{minimax}
\begin{split}
\sigma_{j}=\min\{\|\Phi\mid_{P}\|: \, P \in \mathrm{Grass}_j(V)\}; \ \ \text{and} \\
\sigma_{j+1}=\max\{\m(\Phi\mid_{U}): \, U \in \mathrm{Grass}_{d-j}(V)\},
\end{split}
\end{align}
where $\mathrm{Grass}_r(V)$ denotes the $r$-dimensional Grassmannian of $V$. 

\medskip

From now on, we denote by  $\sigma_1(x,n)\leq \cdots \leq \sigma_d(x,n)$ the singular values of $Df^n_x$ and by $\{e_1^n(x),e_2^n(x),\dots,e_d^n(x)\}$ an orthogonal basis of $T_xM$ so that $\|Df^n(e_i^n(x))\|=\sigma_i(x,n)$, for all $x\in M, \, n\geq 1$ and $1\leq i \leq d$. To simplify notation, we omit the dependence on $x$ and write $e_i$ instead of $e_i(x)$.

In order to prove Proposition~\ref{prop:conedom}, we define $E_n(x)$ as the subspace of $T_xM$ spanned by the vectors $e_i^n$'s in the basis such that $\|Df^n(e_i^n)\|=\sigma_i(x,n)< \sigma_{\kappa+1}(x,n)$; and $E^{\perp}_n(x)$ as the orthogonal complement of $E_n(x)$. It should be noted that $E_n(x)$ is at most $\kappa$-dimensional and $E^{\perp}_n(x)$ is the fastest direction of $Df^n$, which means that $\m(Df^n\mid_{E^{\perp}_n(x)})\geq \|Df^n(v)\|$ for every unit vector $v \notin E^{\perp}_n(x)$.

Since the cone-field $\mathscr{C}$ is $(d-\kappa)$-dimensional and $\ker(Df_x^n)\cap \mathscr{C}(x)=\{0\}$ for each $x \in M$ and $n\geq 1$, we obtain that $\dim \ker(Df^n_x) \leq \kappa$ for each $x \in M$ and $n\geq 1$. Otherwise, the intersection between $\ker(Df^n_x)$ and the cone $\mathscr{C}(x)$ would be nontrivial. Moreover, as $\dim \ker(Df^n_x)=r$ implies that $\sigma_i(x,n)=0$ for each $1\leq i\leq r$, we can conclude that $\ker(Df^n_x)$ is contained in $E_n(x)$.

We will prove that for all $n\geq 1$ the subspaces $E_n(x)$ are $\kappa$-dimensional, implying that they are in $\mathrm{Grass}_{\kappa}(T_xM)$. 
After that, we will show that the family $(E_n)_n$ is a Cauchy sequences and, recalling that $\mathrm{Grass}_{\kappa}(T_xM)$ is a compact space with the usual distance \eqref{dist}, we will get the limit $E(x)$ for each $x \in M$. Finally, we will prove that $E$ is $Df$-invariant.

Note that the subspace $E(x)$ which is limit of $E_n(x)$ satisfies that $\ker(Df^n_x) \subseteq E(x)$ for each $n\geq 1$. Indeed, since $\ker(Df^n_x) \subseteq E_n(x)$ and $\ker(Df^j_x) \subseteq \ker(Df^{j+1}_x)$, we have that $\ker(Df^j_x) \subseteq E_n(x)$ for each $1\leq j \leq n$ and, then, we get that $\ker(Df^j_x) \subseteq E(x)$ for each $j\geq 1$. 

To define the direction $F$, we observe that the assumptions
\begin{align*}
Df_x^{\ell}(\mathscr{C}(x))\subseteq \mathrm{int}(\mathscr{C}(f^{\ell}(x))) \ \ \text{and} \ \  \ker(Df_x^n)\cap \mathscr{C}(x)=\{0\},
\end{align*}
for each $x \in M$ and each $n\geq 1$, allow us to claim the following which the proof can be found in \cite{Crovisier-Potrie}.

\begin{claim}\label{subbundle:F}
	For each orbit $(x_i)_i$, there is a $(d-\kappa)$-dimensional subbundle $F$ defined by 	$$F(x_i)=\cap_{n\geq 0} Df^{n\ell}(\mathscr{C}(x_{i-n\ell})).$$
	Moreover, it is  contained in $\mathrm{int}(\mathscr{C}(x_i))$ and is also $Df^\ell$-invariant.
\end{claim}

To complete the proof of Proposition~\ref{prop:conedom}, we will prove that the subbundle $E$ is transversal to the cone-field, implying that $E\cap F=\{0\}$; and, finally, we will show that $E\prec F$.

These arguments will be divided into a series of lemmas. Before stating the first one, we need to introduce some notation.

Since the subbundle $F$, as in Claim~\ref{subbundle:F}, is in the cone-field $\mathscr{C}$ away from the boundary, we can choose $\delta_0>0$ small enough such that for every orbit $(x_i)_i$ the family of $(d-\kappa)$-dimensional cones $\{\mathcal{K}(x_i,\delta_0):i \in \mathbb{Z}\}$, 
$$\mathcal{K}(x_i,\delta_0):=\{(u,v) \in F^{\perp}(x_i)\oplus F(x_i):\|u\|\leq\delta_0\|v\|\},$$
satisfies that $\mathcal{K}(x_i,\delta_0)$ is the contained in interior of $\mathscr{C}(x_i)$ for each $i \in \mathbb{Z}$.


We are getting this alternative family of cones because that family is centered at the direction $F$ along of the orbit $(x_i)_i$ which does not necessarily happen for the cone-field $\mathscr{C}$ and, this will be useful to make the computation.

Now, we state the first lemma toward the proof of Proposition~\ref{prop:conedom}. Roughly speaking, the lemma guarantees that the directions out of the cone $\mathcal{K}(x_i,\delta)$ are dominated by the direction $F$.


\begin{lemma}\label{lemma:gap}
	For every $0<\delta<\delta_0$ and $\varepsilon>0$ small enough, there is an integer $N>0$ such that for every orbit $(x_i)_i$ holds
	\begin{align}\label{cone}
	Df^{N\ell}(\mathcal{K}(x_i,\delta))\subseteq \mathrm{int}(\mathcal{K}(x_{i+N\ell},\varepsilon))\cup\{0\}.
	\end{align}
	Moreover, for every $0<\lambda<1$ and every $0<\delta<\delta_0$ small enough, we can choose $N>0$ such that if $w \notin \mathcal{K}(x_i,\delta)$ with $\|w\|=1$ then at least one of the following holds:
	\begin{enumerate}[label=$\mathrm{(\roman*)}$]
		\item $Df^{N\ell}(w) \in \mathcal{K}(x_{i+N\ell},\delta)$;
		\item or $\|Df^{N\ell}(w)\|\leq \lambda \m(Df^{N\ell}\mid_{F(x_i)})$.
	\end{enumerate}
\end{lemma}

\begin{proof}
	Since $\mathcal{K}(x_i,\delta) \subseteq \mathscr{C}(x_i)$ along of the orbit $(x_i)_i$ and by the definition of $F$, we have that for every $0<\varepsilon<\delta$ there is a large number $N>0$ such that
	$$Df^{N\ell}(\mathcal{K}(x_i,\delta))\subseteq Df^{N\ell}(\mathscr{C}(x_i)) \subseteq \mathcal{K}(x_{i+N\ell},\varepsilon),$$
	for each $i \in \mathbb{Z}$. This proves the first part of the lemma.
	
	In order to prove the second part, let us fix any $\varepsilon$ which $0<\varepsilon<\left(\frac{\delta}{1+\delta}\right)\left(\frac{\delta\lambda}{1+\delta\lambda}\right)$ and $N>0$ satisfying \eqref{cone}. We suppose by contradiction that there is a unit vector $w \notin \mathcal{K}(x_i,\delta)$ such that $$\|Df^{N\ell}(w)\|\geq \lambda \m(Df^{N\ell}\mid_{F(x_i)}) \ \ \text{and} \ \ Df^{N\ell}(w) \notin \mathcal{K}(x_{i+N\ell},\delta).$$
	Then, we will prove that \eqref{cone} does not happen, contradicting the first part of the lemma.
	
	Indeed, write $Df^{N\ell}(w)=au_1+bv_1$ where $u_1 \in F^{\perp}(x_{i+N\ell})$ and $v_1 \in F(x_{i + N\ell})$ are unit vectors such that $|a|\geq \delta |b|$. Then, we can bound $|a|$ from below and $|b|$ from above as follows: 
	$$|a|\left(1+\dfrac{1}{\delta}\right)\geq|a|+|b|\geq \|Df^{ N\ell}(w)\|
	\ \ \text{and} \ \  \|Df^{ N\ell}(w)\|\geq |b|,$$
	which gives
	$$|a| \geq \left(\dfrac{\delta}{1+\delta}\right) \|Df^{ N\ell}(w)\|.$$
	
	Take a unit vector $v \in F(x_i)$ satisfying $\|Df^{N\ell}(v)\|=\m(Df^{N\ell}\mid_{F(x_i)})$ and choose a number $c \in \mathbb{R}$ such that $w'=cw+v$ belongs to the boundary of $\mathcal{K}(x_i,\delta)$, which is contained in $\mathscr{C}(x_i)$. Note that the constant $c$ satisfies $|c|\geq\delta$, since $w$ is a unit vector out of the cone $\mathcal{K}(x_i,\delta)$ and $v$ is a unit vector in $F(x_i)$.
	
	We now decompose $Df^{N\ell}(w')=ca u_1+ cbv_1+\m(Df^{N\ell}\mid_{F(x_i)})v_0$ where $v_0=Df^{N\ell}(v)/\|Df^{N\ell}(v)\| \in F(x_{i+N\ell})$ and, then, we can get that
	\begin{align*}
	\dfrac{\|acu_1\|}{\|bcv_1+\m(Df^{N\ell}\mid_{F(x_i)})v_0\|}&
	\geq\dfrac{|a||c|}{|b||c|+\m(Df^{N\ell}\mid_{F(x_i)})}\\
	&\geq\left(\dfrac{\delta}{1+\delta}\right) \dfrac{\|Df^{N\ell}(w)\||c|}{\|Df^{N\ell}(w)\||c|+\m(Df^{N\ell}\mid_{F(x_i)})}\\
	&\geq \left(\dfrac{\delta}{1+\delta}\right)\left(\dfrac{|c|}{|c|+\lambda^{-1}}\right)\\
	&\geq \left(\dfrac{\delta}{1+\delta}\right) \left(\dfrac{\lambda|c|}{\lambda|c|+1}\right)\\
	&\geq \left(\dfrac{\delta}{1+\delta}\right) \left(\dfrac{\lambda}{\lambda+|c|^{-1}}\right)\\
	&
	\geq \left(\dfrac{\delta}{1+\delta}\right) \left(\dfrac{\lambda}{\lambda+\delta^{-1}}\right)
	\geq \left(\dfrac{\delta}{1+\delta}\right)\left(\dfrac{\delta\lambda}{1+\delta\lambda}\right)>\varepsilon.
	\end{align*}
	Therefore, we obtain that $Df^{N\ell}(w')$ does not belong to $\mathcal{K}(x_{i+N\ell},\varepsilon)$ which contradicts \eqref{cone} and therefore completes the proof of the lemma.
\end{proof}



From now on, we fix $\lambda, \delta$, and $\varepsilon$ as in Lemma~\ref{lemma:gap} and, to simplify notation, we will assume that $N=1$.

We will prove that there is an ``exponential gap'' which means the quotient between $\sigma_{\kappa}(x,n)$ and $\sigma_{\kappa+1}(x,n)$ decrease exponentially fast to zero. It will be useful to prove that $(E_n(x))_n$ is a Cauchy sequence.

\begin{lemma}\label{l.domsingular}
	There is $c>0$ such that $$\sigma_{\kappa}(x,n) < c \lambda^n \sigma_{\kappa+1}(x,n),$$ for all $x \in M$ and $n\geq 1$.
\end{lemma}

\begin{proof}
	We will first prove the lemma for $n$ multiple of $\ell$.
	
	We will prove for an orbit $(x_i)_i$ that 
	\begin{align}\label{claim-1}
	\sigma_{\kappa}(x_i,n\ell)\leq \lambda^n \sigma_{\kappa+1}(x_i,n\ell),
	\end{align}
	for every $n\geq 1$.
	
	Consider $P_0=E_{n\ell}(x_i)$ and $P_j=Df^{j\ell}(E_{n\ell}(x_i))$ for each $1\leq j \leq n$. We will firstly show that $P_j \cap \mathcal{K}(x_{i+j\ell},\delta) =\{0\}$ for each $0\leq j \leq n$ and, by invariance of the family $\{\mathcal{K}(x_i,\delta):i \in \mathbb{Z}\}$, it is enough to show this for $P_n$. Note that $P_n$ is orthogonal to $Df^{n\ell}(E_{n\ell}^\perp(x_i))$ by definition and, since $\dim E_{n\ell}^\perp(x_i)= \dim Df^{n\ell}(E_{n\ell}^\perp(x_i))$ is at least the dimension of the cone, it is enough to show that $Df^{n\ell}(E_{n\ell}^\perp(x_i)) \subset \mathcal{K}(x_{i+n\ell},\delta)$. If this were not the case, we would find a unit vector $v \in E_{n\ell}^{\perp}(x)$ which does not belong to the cone $\mathcal{K}(x_i,\delta)$ such that $Df^{n\ell}(v)\notin \mathcal{K}(x_{i+n\ell},\delta)$. We have that $\|Df^{n\ell}(v)\|\geq \sigma_{\kappa+1}(x_i,n\ell)\geq \m(Df^{n\ell}\mid_{F(x_i)})$, by definition of $E_{n\ell}(x_i)^{\perp}$. Then, a computation similar to the one in Lemma \ref{lemma:gap} gives the contradiction that $Df^{n\ell}(\mathcal{K}(x_i,\delta))$ is not contained in $\mathcal{K}(x_{i+n\ell},\varepsilon))$.
	
	To prove \eqref{claim-1}, we use the ``minimax'' characterization in \eqref{minimax} which says
	\begin{align*}
	\dfrac{\sigma_{\kappa}(x_i,n\ell)} {\sigma_{\kappa+1}(x_i,n\ell)}
	=\dfrac{\min\{\|Df^{n\ell}\mid_{P}\|:P\in \mathrm{Grass}_{\kappa}(T_{x_{i}}M)\}}
	{\max\{\m(Df^{n\ell}\mid_{U}):U \in \mathrm{Grass}_{d-\kappa}(T_{x_{i}}M)\}}\\
	\leq \dfrac{\|Df^{n\ell}\mid_{E_{n\ell}(x_{i})}\|} {\m(Df^{n\ell}\mid_{F(x_i)})}\leq  \Pi_{j=0}^{n-1}\dfrac{\|Df^{\ell}\mid_{P_j}\|} {\m(Df^{\ell}\mid_{F(x_{i+j\ell})})}.
	\end{align*}
	
	Thus, for every $w \in P_j$, we have that $Df^{\ell}(w) \notin \mathcal{K}(x_{i+j\ell},{\color{blue} \delta})$ and, by Lemma~\ref{lemma:gap}, we obtain that
	$\|Df^{\ell}\mid_{P_j}\|\leq \lambda \m(Df^{\ell}\mid_{F(x_{i+j\ell})})$.
	
	Finally, by continuity of $\mathscr{C}$ and by $\ker(Df^j)\cap \mathscr{C}=\{0\}$, we can get a uniform constant $c>0$ such that
	$$\frac{\max\{\|Df^j_x\|:x \in M\}}{\min\{\m(Df^j_{x}\mid_{U}):U\in \mathrm{Grass}_{d-\kappa}(T_x) \ \ \text{and} \ \ U \subseteq \mathscr{C}(x)\}}<c$$
	for each $1\leq j<\ell$
	and, then, conclude that
	$\sigma_{\kappa}(x_i,n)\leq c \lambda^n \sigma_{\kappa+1}(x_i,n)$.
\end{proof}

The following result follows as immediate consequence of the previous lemma, since $\sigma_{\kappa}(x,n)< \sigma_{\kappa+1}(x,n)$ for every $n$ large enough and $x \in M$.

\begin{corollary}\label{cor:dim}
	We have that $E_n(x)$ is $\kappa$-dimensional, for every $n$ large enough independently on $x \in M$. Consequently, $E^{\perp}_n(x)$ is $(d-\kappa)$-dimensional.
\end{corollary}


Let us remind the reader what we have done so far in order to prove Proposition~\ref{prop:conedom}. By Corollary~\ref{cor:dim}, we can suppose without loss of generality that $E_n(x)$ is $\kappa$-dimensional for all $n\geq 1$ and we will next show that the sequence $(E_n(x))_n$ is of Cauchy on $\mathrm{Grass}_{\kappa}(T_xM)$ endowed with the distance \eqref{dist}.

We now state the following result (see \cite[Lemma A.4]{BPS} and also see \cite{BG,QTZ,Z} for related results). 

\begin{lemma}\label{cauchy sequence}
	The sequence $(E_n(x))_n$ is of Cauchy.
\end{lemma}

Before the proof of the lemma, we first need two useful lemmas.

\begin{lemma}\label{lemma:n}
	For large $n$, we have that $Df^n(E_n^{\perp}(x))\subseteq \mathscr{C}(f^n(x))$ for every $x \in M$.
\end{lemma}

\begin{proof}
	The proof will follow from the item (ii) in Lemma~\ref{lemma:gap}. Indeed, suppose by contradiction that $Df^n(E_n^{\perp})$ is not contained the cone. In particular, there is a vector $v \in E_n^{\perp}(x)$ such that $Df^{j\ell}(Df^r(v))$ does not belongs to the cone with $n=j\ell + r$ and $0\leq r \leq \ell-1$ satisfying:
	\begin{align*}
	\|Df^n(v)\|\leq \lambda^j\m(Df^{j\ell}\mid_{F(x)})\|Df^r(v)\|.
	\end{align*}
	Thus, since $j$ goes to $+\infty$ as $n$ goes to $+\infty$, we can get $n$ large enough such that $\|Df^n(v)\|<\m(Df^{j\ell}\mid_{F(x)})$. Contradicting the definition of $E_n^{\perp}(x)$ to be the fastest direction of $Df^n$. 
\end{proof}

\begin{lemma}\label{ineq}
	There is a constant $K_1>0$ such that for every $x \in M$ and $n$ large the following inequality holds:
	\begin{align*}
	\min\{\|Df_{f^n(x)}\|\sigma_{\kappa+1}(x,n), \|Df_x||\sigma_{\kappa+1}(f(x),n)\} \geq \sigma_{\kappa+1}(x,n+1)\geq K_1\sigma_{\kappa+1}(x,n).
	\end{align*}
\end{lemma}

\begin{proof}
	We take the constant $K_1=\min\{\|Df\mid_{\mathscr{C}(x)}\|: x \in M\}$. Since the cone-field $\mathscr{C}$ depends continuous on $x$ and is transversal to $\ker(Df^n_x)$ for all $n\geq 1$, we have that $K_1>0$.
	
	We will first prove that 
	\begin{align*}
	\|Df_{f^n(x)}\|\sigma_{\kappa+1}(x,n)\geq \sigma_{\kappa+1}(x,n+1) \geq K_1\sigma_{\kappa+1}(x,n).
	\end{align*}
	
	For that, we take $V$ as the subspace of $T_xM$ generated by $E_n(x)\cup \{e_{\kappa+1}^n\}$. Then,
	\begin{align*}
	\|Df_{f^n(x)}\|\sigma_{\kappa+1}(x,n)&\geq \|Df_x^{n+1}\mid_V\|\\
	&=\sigma_{\kappa+1}(x,n+1)\\
	&\geq \m(Df^{n+1}_x\mid_{E^{\perp}_n(x)})\\
	&\geq \m(Df_{f^n(x)}\mid_{Df^n(E^{\perp}_n(x))}) \m(Df^n_x\mid_{E^{\perp}_n(x)}).
	\end{align*}
	
	Thus, we conclude the proof using that $Df^n_x\mid_{E^{\perp}_n(x)}$ is contained in $\mathscr{C}(f^n(x))$ given by Lemma~\ref{lemma:n}.
	
	To complete the proof, it is sufficient to show that $\|Df_x\|\sigma_{\kappa+1}(f(x),n)\geq \sigma_{\kappa+1}(x,n+1)$, since we have already known that $\sigma_{\kappa+1}(x,n+1)\geq K_1\sigma_{\kappa+1}(x,n)$.
	
	In order to do that we will use that
	\begin{align*}
	\sigma_j(f(x),n)=\min\{\|(Df^n)^{\ast}\mid_{P}\|: \, P \in \mathrm{Grass}_j(T_{f^{n+1}(x)}M)\}; \ \ \text{and} \\
	\sigma_{j+1}(f(x),n)=\max\{\m((Df^n)^{\ast}\mid_{U}): \, U \in \mathrm{Grass}_{d-j}(T_{f^{n+1}(x)}M)\}.
	\end{align*}
	
	We consider two orthogonal bases $\{e_1^{n},\dots,e_d^{n}\}$ of $T_{f(x)}M$ and $\{w_1^{n},\dots,w_d^{n}\}$ of $T_{f^{n+1}(x)}M$ satisfying $Df^{n}(e_j^{n})=\sigma_j(f(x),n)w_j^{n}$ and $(Df^{n})^{\ast}(w_j^{n})=\sigma_j(f(x),n)e_j^{n}$. After that, taking $P=\{w_1^n,\dots,w_{\kappa+1}^n\}$, we can verify that $\|Df_x\|=\|(Df_x)^{\ast}\|$ and
	\begin{align*}
	\|(Df_x)^{\ast}\|\sigma_{\kappa+1}(f(x),n)&\geq  \|(Df_x)^{\ast}\|\|(Df_{f(x)}^n)^{\ast}\mid_{P}\|\\
	&\geq\|(Df_x^{n+1})^{\ast}\mid_{P}\|\geq\sigma_{\kappa+1}(x,n+1).
	\end{align*}
\end{proof}

Finally, we can prove that $(E_n(x))_n$ is a Cauchy sequence.

\begin{proof}[Proof of Lemma~\ref{cauchy sequence}] 
	Suppose $n<m$ and observe that
	$$dist(E_m(x),E_n(x)) \leq \sum_{j=1}^{m-n} dist(E_{n+j-1}(x),E_{n+j}(x)).$$
	Then, to prove the lemma, it is sufficient to show that
	the distance $dist(E_{n+1}(x),E_n(x))$ tends to zero exponentially fast with $n$ and uniformly in $x$.
	
	We will omit the point $x$ and simply rewrite $E_n$ and $E_n^{\perp}$ throughout the proof. In order to estimate the distance between $E_n$ and $E_{n+1}$, we take $w\in E_n$ as the farthest unit vector from $E_{n+1}$ and decompose $w=u+v$ in a unique way where $u \in E_{n+1}$ and $v \in  E_{n+1}^{\perp}$. Then, we obtain that  $\cos\varangle(E_n,E_{n+1}^{\perp})=\cos\varangle(w,v)=\|v\|$.
	
	On the one hand, since the image of $E_{n+1}$ and $E_{n+1}^{\perp}$ by $Df^{n+1}$ are orthogonal, we have that
	\begin{align*}
	\|Df^{n+1}(w)\|&=\|Df^{n+1}(u)\|+\|Df^{n+1}(v)\|\\
	&\geq \|Df^{n+1}(v)\|\geq \sigma_{\kappa+1}(x,n+1)\|v\|.
	\end{align*}
	On the other hand,
	\begin{align*}
	\|Df^{n+1}(w)\|\leq K \|Df^n(w)\|\leq K\sigma_{\kappa}(x,n),
	\end{align*}
	where $K= \sup_x \|Df_x\| < \infty$. Thus, by Lemmas \ref{ineq} and \ref{l.domsingular}, one concludes that:
	\begin{align*}
	\mathrm{dist}(E_{n+1}(x),E_n(x))&=\cos\varangle (E_{n+1}^{\perp}(x),E_n(x))\\
	&=\|v\|\leq \dfrac{K \sigma_{\kappa}(x,n)}{\sigma_{\kappa+1}(x,n+1)}\\
	&\leq K\dfrac{\sigma_{\kappa}(x,n)}{K_1\sigma_{\kappa+1}(x,n)}\leq \dfrac{Kc}{K_1} \lambda^n.
	\end{align*}
	
\end{proof}

It should be note that $(E_n(x))_n$ is a Cauchy for every $x \in M$. 
Thus, we can define a continuous $\kappa$-dimensional subbundle of $TM$ by
$$E(x)=\lim E_n(x), \forall x \in M.$$

Recall that $\ker(Df_x^j) \subseteq E(x)$ for all $x \in M$ and all $j\geq 1$ since $\ker(Df_x^j)\subseteq E_n(x)$ for each $1\leq j\leq n$. 

We next prove that $E$ is transversal to the cone-field $\mathscr{C}$ and is a $Df$-invariant subbundle. 

\begin{lemma}
	For each $x \in M$, we have that $E(x)\cap \mathscr{C}(x)=\{0\}$ and $Df(E(x))\subseteq E(f(x))$.
\end{lemma}

\begin{proof}
	We first prove that $E(x)\cap \mathscr{C}(x)=\{0\}$. For that, recall that $Df^n(E_n(x))$ and $Df^n(E_n^{\perp}(x))$ are perpendicular and, by Lemma~\ref{lemma:n}, $Df^n(E_n^{\perp}(x))$ is contained in $\mathscr{C}(x)$. Then, the intersection between $Df^n(E_n(x))$ and the cone-field is trivial, implying that $E_n(x)\cap \mathscr{C}(x)=\{0\}$. Thus, we can conclude that the intersection between $E(x)$ and $\mathscr{C}(x)$ is also trivial.
	
	To show the invariance property, we will prove that the distance between $Df(E_{n+1}(x))$ and $E_n(f(x))$ goes to zero as $n$ tends to infinity. Before, we fix $\rho>0$ such that
	$$\|Df(v)\|\geq \rho \|v\|$$ for every $v \in \ker(Df_x)^{\perp}$ for every $x \in M$.
	
	We start the proof of invariance, taking $v \in E(x)$ such that $Df(v)=0$. Then, clearly, $Df(v) \in E(f(x))$.
	
	Now, we take a unit vector $v \in \ker(Df_x)^{\perp}\cap E(x)$ and choose a sequence of unit vectors $v_n \in E_{n+1}(x)\cap \ker(Df_x)^{\perp}$ such that $v_n \to v$. Then, we write $Df(v_n)=u_n+z_n$ where $u_n \in E_n(f(x))$ and $z_n \in E_n^{\perp}(f(x))$.
	
	Consider two orthogonal bases
	$\{e_1^n,\dots,e_d^n\}$ and $\{w_1^n,\dots,w_d^n\}$ of $T_{f(x)}M$ and  $T_{f^{n+1}(x)}M$, respectively, so that, for each $1\leq i \leq d$, $Df_{f(x)}^n(e_i^n)=\sigma_i(f(x),n) w_i^n$ and $(Df_{f(x)}^n)^{\ast}(w_i^n)=\sigma_i(f(x),n) e_i^n$. Take a vector $\tilde{z}_n \in \mathrm{span}\{w_{\kappa+1}^n,\cdots, w_d^n\}$ such that $(Df_{f(x)}^n)^{\ast}(\tilde{z}_n)=z_n$ and, then, 
	$\|z_n\|\geq \sigma_{\kappa+1}(f(x),n)\|\tilde{z}_n\|$. Thus, fixing $K=\sup \|Df_x\|$ and using Lemma~\ref{ineq}, we obtain that
	\begin{align*}
	\cos\varangle(\mathbb{R}Df(v_n),z_n)&=\dfrac{\langle Df_x(v_n),z_n \rangle}{\|Df_x(v_n)\|\|z_n\|}\\
	&=\dfrac{\langle Df_x(v_n), (Df_{f(x)}^n)^{\ast}(\tilde{z}_n)\rangle}{\|Df_x(v_n)\|\|z_n\|}\\
	&=\dfrac{\langle Df_x^{n+1}(v_n),\tilde{z}_n\rangle}{\|Df_x(v_n)\|\|z_n\|}\\
	& \leq \dfrac{\|Df_x^{n+1}(v_n)\|\|\tilde{z}_n\|}{\|Df_x(v_n)\|\|z_n\|}\\
	&\leq\dfrac{\sigma_{\kappa}(x,n+1)}{\sigma_{\kappa+1}(f(x),n)} \cdot\dfrac{\|v_n\|}{\|Df(v_n)\|}.
	\end{align*}
	Then, by Lemmas~\ref{lemma:gap} and \ref{ineq}, we obtain that
	\begin{align*}
	\cos\varangle(\mathbb{R}Df(v_n),z_n)&\leq
	\dfrac{\sigma_{\kappa}(x,n+1)}{\sigma_{\kappa+1}(f(x),n)} \cdot\dfrac{\|v_n\|}{\|Df(v_n)\|}\\
	&\leq \dfrac{\sigma_{\kappa}(x,n+1)}{\sigma_{\kappa+1}(x,n+1)/K}\cdot \dfrac{1}{\rho}
	\leq \dfrac{1}{\rho} Kc\lambda^{n+1}.
	\end{align*}
	
	Since $\cos\varangle(\mathbb{R}Df(v_n), E_n^{\perp}(f(x)))\leq \cos\varangle(\mathbb{R}Df(v_n),z_n)$, we have that $Df(v_n)$ and $E_n^{\perp}(f(x))$ are getting perpendicular when $n$ tends to infinity. Thus, $Df(v)=\lim Df(v_n)$ is perpendicular to $E(f(x))^{\perp}$ and, consequently, $Df(v)$ belongs to $E(f(x))$.
\end{proof}

So far, we have showed that $E\oplus F$ is nontrivial splitting on $TM_f$. Finally, we will next show that it is a dominated splitting for $f$ and so prove Proposition~\ref{prop:conedom}.

\begin{proof}[Proof of Proposition~\ref{prop:conedom}]
	We have proved so far that under the assumption of Proposition~\ref{prop:conedom} the splitting $E\oplus F$ satisfy for each orbit $(x_i)_i$ that $Df(E(x_i)) \subseteq E(f(x_i))$ and $Df^{\ell}(F(x_i))=F(f^{\ell}(x_i))$.
	Moreover, as the intersection between $E$ and the cone-field $\mathscr{C}$ is trivial, we conclude by (ii) in Lemma~\ref{lemma:gap} (recall that we are assuming $N=1$) that $$\dfrac{\|Df_{x_i}^{n\ell}\mid_{E(x_i)}\|}{\m(Df_{x_i}^{n\ell}\mid_{F(x_i)})}<\lambda^n.$$
	Thus, $E\oplus F$ is a dominated splitting for $f^{\ell}$.
	
	To complete the proof, we use that $E\oplus Df(F)$ is also a dominated splitting for $f^{\ell}$ and get that $Df(F)=F$ by the uniqueness of dominated splitting (see Lemma~\ref{F:uniqueness}).
\end{proof}

\begin{remark}
	It should be emphasized that the existence of a dominated splitting is an open property in the $C^1$ topology. That is, if $f$ admits an invariant cone-field $\mathscr{C}$ transversal to the kernel, then there exists a neighborhood $\mathcal{U}$ of $f$ in $\mathrm{End}^1(M)$ such that the cone-field $\mathscr{C}$ so is for each $g \in \mathcal{U}$.
\end{remark}

\section{Proof of Theorem \ref{thm A}}\label{sec:proof of thm A}
Let us recall that $f_0: M \to M$ is a robustly transitive endomorphism displaying critical points and $\mathcal{F}_0$ is a subset of $\mathcal{U}_0$ which accumulates on $f_0$ and for each $f \in \mathcal{F}_0$ has $\mathrm{int}(\mathrm{Cr}_{\kappa}(f))\neq \emptyset$.

We next prove that $\Lambda_f$ is dense in $M_f$ for every $f \in \mathcal{F}_0$.

\begin{lemma}\label{lambdaset-non-empty}
	If $f$ is a transitive endomorphism and $\mathrm{int}(\mathrm{Cr}_{\kappa}(f))\neq \emptyset$ then $\Lambda_f$ is a nonempty set. Moreover, $\Lambda_f$ is a dense subset of $M_f$.
\end{lemma}

\begin{proof}
	On the one hand, the transitivity of $f$ on $M$ implies that $f: (x_i)_i \mapsto (f(x_i))_i$ is a transitive homeomorphism on $M_f$ (see \cite[Theorem 3.5.3]{AH}). Then, considering a base of open sets $\{B_n\}_n$ of $M_{f}$, we have that $A_n^{+}=\cup_{j\geq 0}f^j(B_n)$ and $A_n^{-}=\cup_{j\geq 0}f^{-j}(B_n)$ are dense open sets in $M_{f}$, for each $n$. By Baire category theorem, we get that $\cap_n (A_n^{+}\cap A_n^{-})$ is a residual set in $M_f$ which consists of the points which their backward and forward orbits are dense in $M_{f}$. On the other hand, since $\mathrm{Cr}_{\kappa}(f)$ has non-empty interior, we have that any orbit in $\cap_n A_n^{+}\cap A_n^{-}$ visits the interior of $\mathrm{Cr}_{\kappa}(f)$ infinitely many times for the past and the future and, hence, it belongs to $\Lambda_{f}$. Thus, $\Lambda_f$ is nonempty and dense in $M_f$.
\end{proof}

We are now assuming that Theorem~\ref{thm B} holds. Then, we have that there are $\alpha>0$ and an integer $\ell>0$ such that $E_f\oplus F_f$ is an $(\alpha,\ell)$-dominated splitting over $\Lambda_f$ for every $f \in \mathcal{F}_0$. Since $\Lambda_f$ is dense on $M_f$ by Lemma~\ref{lambdaset-non-empty}, we use Proposition~\ref{cont-ext} to conclude that $E_f\oplus F_f$ is an $(\alpha,\ell)$-dominated splitting over $M_f$.  This implies Theorem~\ref{thm A1}, as we have already mentioned.

Now we will prove Theorem A. In order to do this, we will need to push the domination on $f \in \mathcal{F}_0$ to $f_0$. 

By the equivalence in Proposition~\ref{prop:domcone}, we can use the uniformity of the dominated splitting to obtain the cone-field $\mathscr{C}_{E_f}:=\{\mathscr{C}_{E_f}(x, \alpha):x \in M\}$ which its dual cone-field $\mathscr{C}_{E_f}^{\ast}$ is $\ell$-invariant uniformly in $f \in \mathcal{F}_0$, up to increasing the constant $\ell$. Then, we consider a sequence of $f_n \in \mathcal{F}_0$ converging to $f_0$, and we define $E_0(x)$ and $\mathscr{C}_{E_0}^{\ast}(x,\alpha)$ respectively as the limit of $E_n(x)$ and $\mathscr{C}_{E_n}^{\ast}(x,\alpha)$, which satisfy $E_0(x)\cap \mathscr{C}_{E_0}^{\ast}(x,\alpha)=\{0\}$.

Finally, it remains to show that $\mathscr{C}_{E_0}^{\ast}(x,\alpha)$ is invariant and $\ker(Df_0^m)\subseteq E_0(x)$ for every $x\in M$ and $m\geq 1$. For that, we first use that $\mathscr{C}_{E_n}^{\ast}(x,\alpha)$ is $\ell$-invariant for every $n$ to conclude that $\mathscr{C}_{E_0}^{\ast}(x,\alpha)$ is also $\ell$-invariant. Secondly, to show that $\ker(Df^m_0)$ at $x \in M$ is contained in $E_0(x)$, it is enough prove that if $u \in \ker(Df_0^m)$ at $x \in M$ then $u \in E_n(x)$ for large $n\geq 1$ since $E_0(x)$ is the limit of $E_n(x)$. Thus, suppose that $u \notin E_n(x)$ for infinitely many $n$. Considering $E_n\oplus F_n$ as an $(\alpha,\ell)$-dominated splitting for $f_n$, we can conclude that for $k$ large enough $\|Df_n^{k\ell}(u)\|\approx \m(Df_n^{k\ell}\mid_{F_n})$, which is away from zero independently on $n$. This contradicts the fact that $Df_n^{k\ell}(u)$ must have to converge to $Df_0^{k\ell}(u)=0$ for $k\ell\geq m$. \qed

\subsection{Proof of Corollary \ref{cor:topobst}}

Let $M=S^{2n}$ be an even dimensional sphere and a robustly transitive map $f: M \to M$.
Let $E \subseteq TM$ be a $\kappa$-dimensional subbundle, that is, a map $E: M \to \mathrm{Grass}_{\kappa}(TM)$ so that $E(x) \in \mathrm{Grass}_{\kappa}(T_xM)$ for every $x \in M$. Note that  we can associate an Euler class to $E$ which is an element of $H^{\kappa}(M),$ the $\kappa$-th dimensional de Rham cohomology of $M$. Suppose $E\neq TM$ (or equivalently, $\kappa \neq 2n=\dim(M)$), then we have that its Euler class is zero since all intermediate homologies vanish on spheres. On the other hand, we have that $E^\perp$ is another subbundle of dimension $2n-\kappa$ complementary to $E$ in $TM$, then it follows that $E \oplus E^\perp = TM$. Therefore, the Euler class of $TM$ (which is the Euler characteristic of $M$ times the class of the volume form $M$ in $H^{2n}(M)$) is the product of the Euler classes of $E$ and $E^\perp$. Hence, the Euler characteristic of $M$ is zero, but this is a contradiction since the Euler characteristic of an even dimensional sphere is $2$. This argument shows that even dimensional spheres do not admit non-trivial subbundles. This argument is classical and is taken from \cite{AB}, but it can be also found in \cite{Milnor-Stasheff}[Property 9.6]. 

Now, we can apply Theorem \ref{thm A} which implies that if an endomorphism $f$ is robustly transitive, then it cannot have critical points as it would imply the existence of a non-trivial subbundle.  Therefore we conclude that it is a  local diffeomorphism (covering map) on $M$. However, since $M$ is simply connected, the endomorphism $f$ should be a diffeomorphism. We can use now \cite{BDP} to conclude the corollary.  \qed

\section{Proof of Theorem \ref{thm B}}\label{sec:proof of thm B}

Here, let $f_0$ be an endomorphism displaying critical points and $\mathcal{F}$ be a family of endomorphisms converging to $f_0$ satisfying:
\begin{align}\label{eq:assumption}
\Lambda_f\neq \emptyset \ \ \text{and} \ \ \dim \ker(Df^m)\leq \kappa, \, \forall f \in \mathcal{F} \ \ \text{and} \ \ m\geq 1. 
\end{align}

First, let us recall that for each $f \in \mathcal{F}$ and for each $x_i$ in the orbit $(x_i)_i \in \Lambda_f$ we have defined the candidate to dominated splitting by
$$E(x_i)=\ker(Df^{m_f+\tau_i^{+}}_{x_i}) \ \ \text{and} \ \ F(x_i)=\im(Df^{|\tau_i^{-}|}_{x_{i+\tau_i^{-}}}),$$
where 
\begin{align*}
\begin{array}{c}
\tau^{+}_i=\min\{n\geq 0:x_{i+n} \in \mathrm{Cr}_{\kappa}(f)\}, \ \  \text{and} \ \
\tau^{-}_i=\max\{n\leq -m_f:x_{i+n} \in \mathrm{Cr}_{\kappa}(f)\}.
\end{array}
\end{align*}

The next statement will be useful to prove that $E$ and $F$ are invariant subbundles over $\Lambda_f$.

\begin{lemma}\label{lemma:1}
	If $x \in \mathrm{Cr}_{\kappa}(f)$ and $n\geq m_f$ then $\ker(Df^n_x)$ and $\im(Df^n_x)$ are $\kappa$- and $(d-\kappa)$-dimensional, respectively. Moreover, if $x_j=f^j(x)$, for each $0\leq j\leq n$ with $n\geq m_f$, such that $x_0,\, x_n \in \mathrm{Cr}_{\kappa}(f)$ then $T_{x_j}M=\ker(Df^{m_f+{n-j}}_{x_j})\oplus \im(Df^j_{x_0})$, for each $m_f\leq j \leq n$.
\end{lemma}

\begin{proof}
	By the assumption \eqref{eq:assumption} and $Df^{m+n}_{y}=Df^{m}_{f^n(y)} \cdot Df^n_{y}$, we can get that
	\begin{align}\label{eq:inv}
	\max\{\dim \ker(Df^{n}_{y}), \dim \ker(Df^{m}_{f^n(y)})\}\leq \dim \ker(Df^{m+n}_y)\leq \kappa.
	\end{align}
	Then, when $y$ or $f^n(y)$ belongs to $\mathrm{Cr}_{\kappa}(f)$, we have that $$\dim \ker(Df^{m_f+n}_y)=\kappa \ \ \text{for all} \ \ n\geq 0.$$
	In particular, $\ker(Df^m_x)$ is $\kappa$-dimensional for every $x \in \mathrm{Cr}_{\kappa}(f)$ and $m\geq m_f$, and so, $\im(Df^m_x)$ is $(d-\kappa)$-dimensional by the theorem of kernel and image. To prove the second part, we assume that there exists a non-zero vector $v \in T_{x_0}M$ such that $u=Df^j_{x_0}(v)$ is a non-zero vector in $\ker(Df^{m_f+n-j}_{x_j})\cap \im(Df^j_{x_0})$. Note that $v \notin \ker(Df^j_{x_0})$ and $v \in \ker(Df^{m_f+n}_{x_0})$. It contradicts the assumption \eqref{eq:assumption} since $\ker(Df^{m_f+n}_{x_0})$ contains $\ker(Df^{j}_{x_0})\oplus \mathrm{span}\{v\}$ that is $(\kappa+1)$-dimensional for any $m_f \leq j \leq n$.
\end{proof}

\begin{invproplemma}\label{lemma:invariance}
	The following statements hold:
	\begin{enumerate}[label=$\mathrm{(\roman*)}$]
		\item the subbundles $E$ and $F$ are $\kappa$ and $(d-\kappa)$-dimensional, respectively;
		\item for every orbit $(x_i)_i \in \Lambda_f$, one has that $T_{x_i}M=E(x_i)\oplus F(x_i), \forall i \in \mathbb{Z}$;
		\item the subbundles $E$ and $F$ are invariant.
	\end{enumerate}
\end{invproplemma}

\begin{proof}
	The items (i) and (ii) follow from Lemma~\ref{lemma:1} since for each $x_i$ in the orbit $(x_i)_i$, we have that $x_{i+\tau_i^{-}}, x_{i+\tau_i^{+}}$ belong to $\mathrm{Cr}_{\kappa}(f)$ and $|\tau_i^{-}|\geq m_f$.
	
	In order to prove the item (iii), we shall observe that $\tau_{i}^{-}= \tau_{i+1}^{-}+1$ and $|\tau_{i+1}^{-}|=|\tau_{i}^{-}|+1$ if $x_{i+1-m_f} \notin \mathrm{Cr}_{\kappa}(f)$. Thus, $$Df(F(x_i))=Df(\mathrm{Im}(Df_{x_{i+\tau_i^{-}}}^{|\tau_i^{-}|}))=\mathrm{Im}(Df_{x_{i+1+\tau_{i+1}^{-}}}^{|\tau_i^{-}|+1})=F(x_{i+1}).$$
	On the other hand, if $x_{i+1-m_f} \in \mathrm{Cr}_{\kappa}(f)$ then $\tau_{i+1}^{-}=-m_f$ and, hence, $\tau_i^{-}<~\tau_{i+1}^{-}$. By Lemma~\ref{lemma:1}, we have that 
	$$T_{x_{i+1+\tau_{i+1}^{-}}}M=\ker(Df^{|\tau_{i+1}^{-}|}_{x_{i+\tau_{i+1}^{-}}})\oplus \im(Df^{|\tau_i^{-}|+1-m_f}_{x_{i+\tau_i^{-}}})
	$$
	since $x_{i+1+\tau_{i+1}^{-}}=x_{i+1-m_f}=f^{|\tau_i^{-}|+1-m_f}(x_{i+\tau_i^{-}})$. Therefore, $$Df(F(x_i))=\mathrm{Im}(Df_{x_{i+\tau_i^{-}}}^{|\tau_i^{-}|})=Df^{m_f}(\im(Df^{|\tau_i^{-}|+1-m_f}_{x_{i+\tau_i^{-}}}))=F(x_i).$$
	
	The invariance of $E$ follows from the assumption \eqref{eq:assumption} and from the fact that $E$ is $\kappa$-dimensional.
\end{proof}

The following lemma will be often used throughout this paper.

\begin{lemma}\label{Franks-version}
	There exists a small real number $\varepsilon >0$ such that if $f \in \mathcal{U}$ whose $\Lambda_f \neq \emptyset$ then for every finite collection $\Sigma$ in $(x_i)_i \in \Lambda_f$, say $\Sigma=\{x_0,...,x_{n-1}\}$, and every collection of linear maps $L_i:T_{x_i}M \to T_{f(x_i)}M$ for $i=0,1,\dots,n-1$, satisfying the following two conditions:
	$$\exists v \in F(x_0) \ \ \text{such that} \ \ L_n\cdots L_1(v)\in E(x_n) \ \ \text{and} \ \ \|L_i-Df_{x_{i-1}}\|<\varepsilon,$$
	there exists an endomorphism $\hat{f} \in \mathcal{U}$ and a neighborhood $B$ of $\{x_0,\dots,x_{n-1}\}$ so that
	$\dim \ker(D\hat{f}^m)\geq k+1$ for some positive integer $m$.
\end{lemma}

\begin{proof}
	Let $\varepsilon >0$ be given by \hyperref[lemma:Franks]{Franks' Lemma}. Then, up to shrinking the neighborhood $\mathcal{U}$
	Let us recall the orbit $(x_i)_i \in \Lambda_f$ hits on $\mathrm{Cr}_{\kappa}(f)$ infinitely many times for the past and for the future. Hence, assume, without loss of generality, that $\tau_0^{-}$ and $\tau_0^{+}$ defined for the orbit $(x_i)_i$ satisfy $\tau_0^{-}\leq 0 < n\leq \tau_0^{+}$. Also by \hyperref[lemma:Franks]{Franks' Lemma}, there exists $\hat{f} \in \mathcal{U}$ such that the neighborhood $B$ of $\Sigma$ such that verifies
	$$B\cap\{x_{\tau_0^{-}},\dots,x_0,\dots,x_{n-1},\dots,x_{\tau_0^{+}}\}=\Sigma.$$
	Then, there is $w \in F(x_{\tau_0^{-}})$ such that $Df^{|\tau_0^{-}|}(w)=v$, and taking $m=m_f+\tau_0^{+}+|\tau_0^{-}|$ we get that
	\begin{align*}
	D\hat{f}^m(w)&=Df^{m_f+\tau_0^{+}+-n}L_n\cdots L_1 Df^{|\tau_0^{-}|}(w)\\
	&=Df^{\tau_0^{+}-n}L_n\cdots L_1(v)\in Df^{m_f+\tau_0^{+}-n}(E(x_n))=Df^{m_f}(E(x_{\tau_0^{+}}))=\{0\}.
	\end{align*}
	Therefore, $E(x_{\tau_0^{-}})\oplus \mathrm{span}\{ w\}$ is contained in $\ker(D\hat{f}^m)$. This completes the proof.
\end{proof}

From now on, to prove Theorem~\ref{thm B}, we suppose that there is a  neighborhood $\mathcal{U}_0$ of $f_0$ such that
\begin{align}\label{assumption}
\dim \ker(Df^m)\leq \kappa, \, \forall f\in \mathcal{U}_0 \ \ \text{and} \ \ m \geq 1 .
\end{align}
In other words, we are assuming that the second item of Theorem~\ref{thm B} cannot happen. 

Furthermore, we can also assume that:
\begin{align}
\sup_{(x,v) \in TM, \|v\|=1}\|Df_x(v)\| \leq K, \forall f \in \mathcal{U}_0.
\end{align}

Throughout the rest of this paper, we fix $\varepsilon_0 >0$ as in Lemma~\ref{Franks-version}, and choose $\alpha>0$ the angle small enough so that for all rotation $R$ on $T_{f(x)}M$ of angle smaller than $\alpha$ satisfies:
\begin{align}\label{rotation}
\|R\circ Df_x -Df_x\|<\varepsilon_0, \forall f \in \mathcal{U}_0.
\end{align}

We would like the reader to keep in mind that the existence of a family of rotations $R_1,\dots,R_n$ of angles smaller than $\alpha$ such that $L_i=R_i\circ Df_{x_{i-1}}$, for $1\leq i \leq n$, satisfying the hypothesis of Lemma~\ref{Franks-version} is incompatible with \eqref{assumption}. In the course of this section, we will evoke this incompatibility. 

We now are devoted to prove the first item in Theorem~\ref{thm B}. 
More precisely, we will prove there are $\ell>0$ and $\alpha>0$ such that $E\oplus F$ is an $(\alpha,\ell)$-dominated splitting over $\Lambda_f$ for every $f \in \mathcal{F}$.

\begin{anglelemma}\label{lemma-angles}
	For every $f \in \mathcal{F}$ and each $(x_i)_i \in \Lambda_f$ holds that $$\varangle(E(x_i),F(x_i))\geq \alpha, \, \forall i \in \mathbb{Z}.$$
\end{anglelemma}

\begin{proof}
	Otherwise, there is a unit vector $w \in F(f(x))$, for some $x$ in $(x_i)_i \in \Lambda_f$, so that $\varangle(\mathbb{R} w, E(f(x)))< \alpha$. Thus, we take a rotation $R: T_{f(x)}M \to T_{f(x)}M$ of angle $<\alpha$ such that $R(w) \in E(f(x))$ which contradicts the assumption \eqref{assumption} since $\Sigma=\{f(x)\}$ and $L_1=R\circ Df_x$ satisfy the hypothesis of Lemma~\ref{Franks-version}.
\end{proof}

To conclude the proof of Theorem~\ref{thm B} remains to show the following.

\begin{dominatelemma}
	There exists an integer $\ell>0$ such that for every $f \in \mathcal{F}$, the splitting $E\oplus F$ over $\Lambda_f$ satisfies $E\prec_{\ell} F$.
\end{dominatelemma}

The proof of this lemma, in some sense, carries the same ideas used in the context of diffeomorphisms, where the lack of domination property allows us to create a point exhibiting some obstruction for the robust transitivity. However, for endomorphisms displaying critical points, the proof is more subtle.

Let us make some comments in order to be clear how we are proceeding with the proof of the domination property lemma. We will assume by contradiction that the splitting $E\oplus F$ does not 
satisfy the domination property and, hence, there are some unit vectors $v \in E(x)$ and $w \in F(x)$ so that for some point $x$ and large positive integer $\ell$ the following holds
$$\|Df^j(v)\|\geq \frac{1}{2} \|Df^j(w)\|,$$
for each $1\leq j \leq \ell$. Note that, as $Df^j(w)\neq 0$, we have that $Df^j$ restricts to 
the bidimensional subspace $V_0=\mathrm{span}\{v,w\}$ is an isomorphism onto its image, $V_j\subseteq T_{f^j(x)}M$. Let us call $Df\mid_{V_{j-1}}:V_{j-1} \to V_j$ by $A_j$. Then, we provide some control over the norms of $A_j$ and $A_j^{-1}$ for each $1\leq j\leq \ell$ in order to apply Lemma~\ref{lemma-matrices} to mix the subbundles $E$ and $F$ to increase the dimension of the critical set, which is incompatible with \eqref{assumption}. However, the control over the norms of $A_j$ and $A_j^{-1}$ for each $1\leq j \leq \ell$ is the subtle part of proof since $Df^{-1}$ is not necessarily defined due to the existence of critical points.

To obtain such norm control of $Df^{-1}$, we divide our approach in two cases. The first one, we assume that $\dim \ker(Df_0)=r<\kappa$. In this case, up to shrink the neighborhood of $f_0$, we have that $\dim \ker(Df)\leq r$ for every $f \in \mathcal{U}_0$. Then, we can find a large $N>0$ such that
\begin{align}\label{ineqb}
\frac{1}{N}\leq \|Df_x|_{\Gamma}\|\leq N,
\end{align}
where for all $x \in M$ and all $\kappa$-dimensional subspace $\Gamma$ of $T_xM$. Then, we can conclude that $1/N \leq  \|A_j\|,\|A_j^{-1}\|\leq N$ for each $1\leq j \leq \ell$. 

The second case, we assume that $\dim\ker(Df_0)=\kappa$. 
Then,
\begin{center}
	$\mathrm{Cr}_{\kappa}(f_0)=\{x \in M: \dim\ker({Df_0}^{m_{f_0}}_x)=\kappa\}$ is non-empty and $m_{f_0}=1$.
\end{center}

This allows us to conclude that every $f \in \mathcal{F}$ satisfies that $m_f=1$ since $m_f\leq m_{f_0}$. Then, to obtain the control of $\|A_j\|$ and $\|A_j^{-1}\|$, we will show that there is a neighborhood $U$ of $\mathrm{Cr}_{\kappa}(f_0)$ in $M$ such that if $f \in \mathcal{F}$ then $\mathrm{Cr}_{\kappa}(f) \subseteq U$ and the piece of the orbit $\{f^j(x):1\leq j \leq \ell\}$ along which there is lack of domination property is away from $U$. This will imply a similar inequality as in \eqref{ineqb} since $\dim \ker(Df)<\kappa$ out of $U$.

The next lemma deals with this subtle part of the proof.

\begin{lemma}\label{lemma-domination}
	Assume that $\dim \ker(Df_0)=\kappa$. Then, there exists a neighborhood $\mathcal{U} \subseteq \mathcal{U}_0$ of $f_0$ and another one $U$ of $\mathrm{Cr}_{\kappa}(f_0)$ such that for any $f \in \mathcal{F}\cap \mathcal{U}$ has its critical point $\mathrm{Cr}_{\kappa}(f)$ is contained in $U$ and if $(x_j)_j \in \Lambda_f$ satisfies that
	\begin{align}\label{eq-lemma}
	\|Df^j\mid_{E(x_0)}\|\geq \frac{1}{2} \m(Df^j\mid_{F(x_0)}),
	\end{align}
	for each $1\leq j \leq l$, then  $x_0,x_1,...,x_{l-1} \notin U$. 
\end{lemma}

\subsection{Domination property near the critical points}
Our goal here is to prove Lemma~\ref{lemma-domination}. For that, we will first make some preliminaries that will be needed in the proof.

Denote by $0\leq \sigma_{1}(x)\leq \sigma_{2}(x) \leq \cdots \leq \sigma_d(x)$ the singular values of $Df$ at $x \in M$. Using that $\mathcal{F}$ accumulates at $f_0$ (in $C^1$-topology) together with the fact that $\ker(Df_0)$ at $x$ is $\kappa$-dimensional for each $x \in \mathrm{Cr}_{\kappa}(f_0)$, we can find a neighborhood $\mathcal{U}_0$ of $f_0$ and a neighborhood $U_0$ of $\mathrm{Cr}_{\kappa}(f_0)$ so that $\mathrm{Cr}_{\kappa}(f)$ is contained in $U_0$ for every $f$ in $\mathcal{F}\cap \mathcal{U}_0$. In particular, the singular values of $Df$ at $x \in \mathrm{Cr}_{\kappa}(f)$ are the following:
$$0=\sigma_{1}(x)=\cdots=\sigma_{\kappa}(x)< \sigma_{\kappa+1}(x)\leq \cdots \leq \sigma_d(x).$$ 
By continuity of $\sigma_i$'s on $M$ since $f$ is a $C^1$-endomorphism, we can assume that for all $x \in U_0$ one has that
\begin{align}\label{inequality}
\sigma_{1}(x)\leq \dots \leq\sigma_{\kappa}(x)< \sigma_{\kappa+1}(x)\leq \dots \leq \sigma_d(x).
\end{align}

Denote by $V_f(x)$ and $W_f(x)$ the subspaces of $T_xM$ generated by the eigenvectors associated to $\sigma_1(x),\dots , \sigma_{\kappa}(x)$ and $\sigma_{\kappa+1}(x), \dots, \sigma_d(x)$, respectively. These subspaces satisfy the following properties:
\begin{enumerate}[label=(\arabic*)]
	\item $V_f(x)\oplus W_f(x)=T_xM$, for every $x \in U_0$; \label{eq-1}
	\item  $V_{f}(x)\perp W_{f}(x)$ and $Df(V_{f}(x))\perp Df(W_{f}(x))$, for every $x \in U_0$; \label{eq-2}
	\item If $f \in \mathcal{F}$ and $(x_i)_i \in \Lambda_f$ with $x_i \in \mathrm{Cr}_{\kappa}(f)$ then $V_f(x_i)=\ker(Df_{x_i})=E(x_i)$ and $Df_{x_i}(W_f(x_i))=\mathrm{Im}(Df_{x_i})=F(x_i)$. \label{eq-3}
\end{enumerate}
Furthermore, it follows from the continuous dependence on the singular values of $Df$ at $x$ that the following maps are continuous:
\begin{align}\label{cont-nbhd}
\mathcal{U}_0\times U_0 \ni (f,x) \mapsto V_f(x) \ \ \text{and} \ \ \mathcal{U}_0\times U_0 \ni (f,x) \mapsto W_f(x).
\end{align}

\begin{remark}
	In \eqref{cont-nbhd}, the continuity is seen in $\mathrm{Grass}_r(TM)$ the Grassmannian of the vector bundle $TM$, which its points are all the pairs $(x,\Gamma)$ such that $\Gamma$ is a $r$-dimensional subspace of $T_xM$, for $r=\kappa$ and $r=d-\kappa$, respectively.
\end{remark}

For every $x \in U_0$ and $\eta \in (0,\pi/2)$, define the following cone-field:
\begin{align*}
\mathscr{C}_{V_f}(x,\eta) &=\{u \in T_xM: \varangle(\mathbb{R}u, W_f(x))\geq \pi/2- \eta\},
\end{align*}
and so $\mathscr{C}^{\ast}_{V_f}(x,\eta)=\{u \in T_xM: \varangle(\mathbb{R}u, V_f(x))\geq \eta\}$. Furthermore, it follows from \ref{eq-2} that $\mathscr{C}_{W_f}(x,\pi/2-\eta)=\mathscr{C}^{\ast}_{V_f}(x,\eta)$.

\begin{lemma}\label{VW-lemma}
	For all $\eta, \theta \in (0,\pi/2)$ small enough, up to shrinking the neighborhoods $\mathcal{U}_0$ and $U_0$, we have that for every $f \in \mathcal{U}_0$ and $x \in U_0$ the following holds: 
	
	For every $u \in \mathscr{C}^{\ast}_{V_f}(x,\eta)$, writing $u=(v,w) \in V_f(x)\oplus W_f(x)$, one has that
	\begin{align}\label{ineq-lemma}
	\varangle(\mathbb{R}Df(u),\mathbb{R}Df(w))< \theta.
	\end{align}
\end{lemma}

\begin{proof}
	We first notice that $V_{f_0}(x)=\ker(Df_0(x))$ and $Df_0(T_xM)=Df_0(W_{f_0}(x))$ for every $x \in \mathrm{Cr}_{\kappa}(f_0)$. Then, for every $u\notin V_{f_0}$ one decomposes $u=(v,w)\in V_{f_0}(x)\oplus W_{f_0}(x)$ with $w\neq 0$, implying that $Df_0(u)=Df_0(w)$. In particular,  $\mathbb{R}Df_0(u)$ is contained in $Df_0(W_{f_0}(x))$ for all $u \in \mathscr{C}^{\ast}_{V_{f_0}}(x,\eta)$ and $x \in \mathrm{Cr}_{\kappa}(f_0)$. 
	
	By the continuity in \eqref{cont-nbhd}, we can shrink the neighborhoods $\mathcal{U}_0$ and $U_0$ to get \eqref{ineq-lemma}.
\end{proof}

The next lemma provides a relation between the splittings $E\oplus F$ and $V\oplus W$ which guarantees a domination property around of a component of the critical set whose the kernel is $k$-dimensional. For simplicity, the set $\mathcal{U}$ will denote a neighborhood of $f_0$ contained in $\mathcal{U}_0$ and $U$ denote a neighborhood of $\mathrm{Cr}_{\kappa}(f_0)$ contained $U_0$.

\begin{lemma}\label{lemma-CCS}
	Given  $\eta, \theta \in (0,\pi/2)$  small enough, there exist neighborhoods $\mathcal{U}$ and $U$ so that for each $f \in \mathcal{F}\cap \mathcal{U}$ and each $(x_i)_i \in \Lambda_f$ one has that if $x_i \in U$ then for all $u_1 \in E(x_i)$ and $u_2 \in F(x_i)$, writing $u_j=(v_j,w_j) \in V_f(x_i)\oplus W_f(x_i)$ for $j=1,2$, the following hold:
	\begin{align}\label{eq-lemma-VW}
	\varangle(\mathbb{R}u_1,\mathbb{R}v_1)\leq\eta \ \ \text{and} \ \ \varangle(\mathbb{R}Df_{x_i}(u_2),\mathbb{R}Df_{x_i}(w_2))
	\leq \theta.
	\end{align}
\end{lemma}

\begin{proof}
	Recall that $\varepsilon_0>0$ and $\alpha >0$ are fixed as in $\eqref{rotation}$. Assume that $2\eta, 2\theta <\min\{\varepsilon_0, \alpha\}/2$. By Lemma~\ref{VW-lemma}, there exist neighborhoods $\mathcal{U}$ and $U$ such that for every $f \in \mathcal{U}$ and $x \in U$ the inequality \eqref{ineq-lemma} holds and, moreover, for every $f \in \mathcal{F}\cap \mathcal{U}$ the component $\mathrm{Cr}_{\kappa}(f)$ is contained in $U$.

	We notice that from the inequality \eqref{ineq-lemma} follows that to prove the lemma it is sufficient to show that for every $(x_i)_i \in \Lambda_f$ whose $x_i \in U$, we have
	\begin{align}\label{eq-proof}
	E(x_i) \subseteq \mathscr{C}_{V_f}(x_i,\eta) \ \ \mathrm{and} \ \ F(x_i) \subseteq \mathscr{C}^{\ast}_{V_f}(x_i,\eta).
	\end{align}
	
	Thus, our current goal is to prove the equation \eqref{eq-proof}. For that, observe that as $\varangle(F(x_i),E(x_i))\geq \alpha>2\max\{\eta,\theta\}$, we have that the vectors of $E(x_i)$ and $F(x_i)$ are neither in $\mathscr{C}_{V_f}(x_i,\eta)$ nor in $\mathscr{C}_{V_f}^{\ast}(x_i,\eta)$ simultaneously. Thus, it is only one of the following can happen:
	\begin{itemize}
		\item either $E(x_i) \subseteq \mathscr{C}_{V_f}(x_i,\eta) \ \ \mathrm{and} \ \ F(x_i) \subseteq \mathscr{C}^{\ast}_{V_f}(x_i,\eta)$;
		\item or  $E(x) \subseteq \mathscr{C}^{\ast}_{V_f}(x_i,\eta) \ \ \mathrm{and} \ \ F(x_i) \subseteq \mathscr{C}_{V_f}(x_i,\eta)$.
	\end{itemize}
	
	If $E(x_i) \subseteq \mathscr{C}^{\ast}_{V_f}(x_i,\eta)$ and $F(x_i) \subseteq \mathscr{C}_{V_f}(x_i,\eta)$ then there is a vector $u \in \mathscr{C}^{\ast}_{V_f}(x_i,\eta)$ which $\varangle(\mathbb{R}u,F(x_i))\leq 2\eta$ and $\varangle(\mathbb{R} Df_x(u),E(f(x_i)))\leq 2\theta$. Since $2\eta,2\theta$ are smaller than $\alpha$, we can take two rotation $R_0$ and $R_1$ of angles less than $\alpha$ and $u' \in F(x_i)$ satisfying $u \in R_0(F(x_i))$ and $R_1(Df_x(u))\in E(f(x_i))$. Recall from the discussion in \eqref{rotation} to conclude that it is incompatible with \eqref{assumption}.
	
	Therefore,  $E(x_i) \subseteq \mathscr{C}_{V_f}(x_i,\eta)$ and $F(x_i) \subseteq \mathscr{C}^{\ast}_{V_f}(x_i,\eta)$ which concludes the proof.
\end{proof}

The previous lemma means that around of $\Lambda$ the subbundle $E$ is $\eta$-close to $V_f$, and that every $(d-\kappa)$-dimensional subspace contained in $\mathscr{C}_{V_f}^{\ast}$, which means it is away from $V_f$, is sent by $Df$ into a subspace $\theta$-close to $Df(W_f)$. This will allows us to state the following. 		

\begin{lemma}\label{constant-C_0}
	For each $\rho>0$ small enough, we can find $\eta, \theta>0$ and the neighborhoods $\mathcal{U}$ and $U$ satisfying Lemmas~\ref{VW-lemma} and \ref{lemma-CCS}  such that for each $f \in \mathcal{F}\cap \mathcal{U}$ and $(x_i)_i \in \Lambda_f$, if $x_i \in U$ then
	$$\|Df\mid_{E(x_i)}\|<\rho \ \ \text{and} \ \ \|Df(v)\|\geq 2\rho\|v\|, \, \forall v \in \mathscr{C}_{V_f}^{\ast}(x_i,\eta)\backslash\{0\}.$$
\end{lemma}

\begin{proof}
	Recall that the maps $\mathcal{U} \times U \ni(f,x) \mapsto V_f(x), \, W_f(x)$ are continuous, where $\mathcal{U}$ and $U$ are neighborhoods of $f_0$ and $\mathrm{Cr}_{\kappa}(f_0)$ respectively. Then, up to shrinking the neighborhoods $\mathcal{U}$ and $U$, we can assume that $V_f$ and $V_{f_0}$ are $\eta$-close to each other and, consequently, $W_f$ and $W_{f_0}$ are both contained in the dual-cone of each other.
	
	Fix $\eta,\theta>0$ small enough and take the neighborhoods $\mathcal{U}$ and $U$ satisfying Lemmas~\ref{VW-lemma} and \ref{lemma-CCS}. We observe that for every $u=(v,w) \in \mathscr{C}^{\ast}_{V_f}(x,\eta)$, with $v \in V_f(x)$ and $w\in W_f(x)$, \begin{align*}
	\tan\eta \leq \tan\varangle(\mathbb{R}u,V_f(x))=\dfrac{\|w\|}{\|v\|},
	\end{align*}
	and, then,
	\begin{align*}
	\|Df(u)\|&\geq \|Df(w)\|-\|Df(v)\|\\
	&\geq \m(Df\mid_{W_f(x)})\|w\|-\|Df\mid_{V_f(x)}\|\|v\|\\
	&\geq \left(\m(Df\mid_{W_f(x)})-\dfrac{\|Df\mid_{V_f(x)}\|}{\tan \eta}\right)\|w\|.
	\end{align*}
	
	We now use that $\left(1+\dfrac{1}{\tan \eta}\right)\|w\|\geq \|w\|+\|v\|\geq \|u\|$ to obtain that
	\begin{align*}
	\|Df(u)\|\geq \left(\dfrac{\tan \eta}{1+\tan \eta}\right)\left(\m(Df\mid_{W_f(x)})-\dfrac{\|Df\mid_{V_f(x)}\|}{\tan \eta}\right)\|u\|.
	\end{align*}
	
	Since $V_{f_0}(x)=\ker(Df_0)$ and $\m(Df\mid_{W_{f_0}(x)})$ is uniformly away from zero for each $x \in \mathrm{Cr}_{\kappa}(f_0)$, we can shrink only the neighborhoods $\mathcal{U}$ and $U$ to obtain that $V_f(x)$ and $W_f(x)$ are close to $V_{f_0}(x)$ and $W_{f_0}(x)$ for every $x \in U$. Then, we can assume that
	$\m(Df\mid_{W_f(x)})$ is uniformly away from zero and $\|Df\mid_{V_f(x)}\|$ is so close to $0$ such that $$\left(\dfrac{\tan \eta}{1+\tan \eta}\right)\left(\m(Df\mid_{W_f(x)})-\dfrac{\|Df\mid_{V_f(x)}\|}{\tan \eta}\right)\geq 2\rho>0$$
	and, hence, $\|Df(u)\|\geq 2 \rho \|u\|$ for every $x \in U$.
	
	Finally, we take $f\in \mathcal{F}\cap \mathcal{U}$ and use that $E(x)$ and $V_f(x)$ depend continuously on $x$ in $U$ and are both identical on $\mathrm{Cr}_{\kappa}(f)$ to conclude that $\|Df\mid_{E(x)}\|\approx \|Df\mid_{V_f(x)}\|$ which is close to $\|Df_0\mid{V_{f_0}}(x)\|\approx 0$ for each $x \in U$. Therefore, up to shrinking the neighborhoods $\mathcal{U}$ e $U$, we can get that $\|Df\mid_{E(x)}\|<\rho$ for every $x \in U$.
\end{proof}

Now we are able to prove Lemma \ref{lemma-domination}.

\subsection*{Proof of Lemma \ref{lemma-domination}}
By Lemma \ref{lemma-CCS}, for all $\theta >0$ and all $\eta> 0$ small, there exist neighborhoods $\mathcal{U}$ and $U$ so that for each $f \in \mathcal{F}\cap\mathcal{U}$ and each $(x_i)_i \in \Lambda_f$ if $x_i \in U$ then the following hold:
\begin{align}\label{eq-cone-EF}
E(x_i)\subseteq \mathscr{C}_{V_f}(x_i,\eta) \ \ \mathrm{and} \ \ F(x_i)\subseteq \mathscr{C}^{\ast}_{V_f}(x_i,\eta).
\end{align}
In particular, $\varangle(Df(W_f(x_i)), \mathbb{R} Df(v))<\theta$ for every $v \in \mathscr{C}^{\ast}_{V_f}(x_i,\eta)$.

We consider the cone-field centered at $F$ of length $\beta>0$ in coordinates $E\oplus F$ over $(x_i)_i \in \Lambda_f$ by:
\begin{align*}
\mathscr{C}_{F,E}(x_i,\beta)=\{(u_1, u_2)  \in E(x_i)\oplus F(x_i):\|u_1\|\leq \beta \|u_2\|\}.
\end{align*}

Choose $\beta>0$ so that $\varangle(\mathbb{R}u,F(x_i))<\alpha$ for every $u \in \mathscr{C}_{F,E}(x_i,\beta)$, where $\alpha>0$ was fixed in \eqref{rotation}.

\begin{remark}
	It should be emphasized that the cone-field above is not as the previous one defined in \S\ref{subsec:cone-criterion}, which the length is the angle, since $E$ and $F$ are not necessary orthogonal.
\end{remark}

Since the angle between $E$ and $F$ is uniformly away from zero, we can assume that $\theta, \eta >0$ above are chosen small enough so that for any $(x_i)_i \in \Lambda_f$, one has that if  $x_i \in U$ then:
\begin{enumerate}[label=(\alph*)]
	\item \label{eq:a}$\forall v \in \mathscr{C}_{V_f}(x_i,\eta) \Longrightarrow \varangle(\mathbb{R}v, E(x_i))<\alpha$; 
	\item \label{eq:b}$\varangle(\mathbb{R}u,F(f(x_i)))<2\theta \Longrightarrow u \in \mathscr{C}_{F,E}(f(x_i),\beta/2)$.
\end{enumerate}

Assume now that $(x_i)_i \in \Lambda$ satisfies the inequality \eqref{eq-lemma}. That is, $(x_i)_i$ satisfies:
\begin{align*}
\|Df^j\mid_{E(x_0)}\|\geq \frac{1}{2}\m(Df^j\mid_{F(x_0)}), \ \ \text{for any} \ \ 1\leq j \leq l.
\end{align*}
In particular, it is immediate from Lemma \ref{constant-C_0} that $x_0 \notin U$.

It remains to prove that $x_1,\dots, x_{l-1} \notin U$. For that, we first observe that $x_1,\dots, x_{l-1}$ are pairwise distinct since, without loss of generality, if $x_l=x_0$ then, as the orbit hit on $\mathrm{Cr}_{\kappa}(f)$ infinitely many time for the future, we have that $x_j \in \mathrm{Cr}_{\kappa}(f)$ for some $\leq 1\leq j\leq l-1$ and, hence, $\|Df\mid_{E(x_j)}\|=0$. This contradicts the fact that $\|Df\mid_{F(x_j)}\|\neq 0$.

Second, we suppose without loss of generality that $x_{l-1} \in U$ and take $v \in E(x_0)$ and $w \in F(x_0)$ the unit vectors whose 
\begin{align*}
\|Df^l(v)\|=\|Df^l\mid_{E(x_0)}\| \geq \frac{1}{2} \m(Df^l\mid_{F(x_0)})=\|Df^l(w)\|.
\end{align*}
Then, we have that $u=(v,\beta^{-1}w) \in \partial\mathscr{C}_{F,E}(x_0,\beta)$ since  $\|v\|=1=\beta\|\beta^{-1}w\|$ and, hence,
\begin{align*}
\|Df^l(v)\|=\|Df^l\mid_{E(x_0)}\|\geq \frac{1}{2}\m(Df^l\mid_{F(x_0)})
=\frac{\beta}{2}\|Df^l(\beta^{-1}w)\|.
\end{align*}
Taking $u_j=Df^j(v,\beta^{-1}w)$ for each $1\leq j \leq l$, we have that $u_l=Df^l(v,\beta^{-1}w)$ belongs to dual cone $\mathscr{C}_{F,E}^{\ast}(x_l,\beta/2)$. This implies that $u_{l-1} \in  \mathscr{C}_{V_f}(x_{l-1},\eta)$. Otherwise, if $u_{l-1} \notin  \mathscr{C}^{\ast}_{V_f}(x_{l-1},\eta)$ then, writing $u_{l-1}=(v',w') \in V_f(x_{l-1})\oplus W_f(x_{l-1})$ and taking $z$ in the intersection between $F(x_{l-1})$ and the plane generated by $\{u_{l-1},w'\}$, one can conclude that $z=av'+bw'$ and, by Lemma~\ref{lemma-CCS}, one gets that
\begin{align*}
\varangle(\mathbb{R} u_l,F(x_l))& \leq \varangle(\mathbb{R} Df(u_{l-1}),\mathbb{R}Df(z))\\
&\leq \varangle(\mathbb{R} Df(u_{l-1}),\mathbb{R}Df(w')) +\varangle(\mathbb{R}Df(w'),\mathbb{R}Df(z))\leq 2\theta,
\end{align*}
that, by \ref{eq:b}, it implies that $u_l \in \mathscr{C}_{F,E}(x_l,\beta/2)$ which is a contradiction.

Thus, by \ref{eq:a} and the choice of $\beta>0$, one concludes that $\varangle(\mathbb{R}u_{l-1}, E(x_{l-1}))< \alpha$ and $\varangle(\mathbb{R}u, F(x_0))< \alpha$ since $u \in \mathscr{C}_{F,E}(x_0,\beta)$. Thus, we can find two rotations $R_0:T_{x_0}M \to T_{x_0}M$ and $R_1:T_{x_{l-1}}M \to T_{x_{l-1}}M$ of angle small than $\alpha$ and $u_0 \in F(x_0)$ such that
$$R_0(u_0)=u, \ \ R_1( u_{l-1}) \in E(x_{l-1}), \ \ \text{and} \ \ \|R_iDf-Df\|<\varepsilon_0,\ \ i=0,1.$$

Finally, we take that $\Sigma=\{x_{-1},\cdots,x_{l-2}\}$ and $L_j:T_{x_{j-2}}M \to T_{x_{j-1}}M$ defined by
$L_1=R_0\circ Df_{x_{-1}},\,  L_2=Df_0, \, \cdots,\,
L_{l-1}= Df_{x_{l-3}},\, \text{and} \ \ L_{l}=R_1\circ Df_{x_{l-2}}$
satisfy the hypothesis of Lemma~\ref{Franks-version} which is incompatible with \eqref{assumption}. \qed

\subsection{Proof of the uniform domination property}

This section is  devoted to conclude the proof of the uniform domination property. First of all, some preliminaries are needed. 

Let us introduce a classical result about lack of domination of matrices which some proofs can be seen in \cite[Appendix A]{Potrie} and \cite[Lemma 3.1]{BM}. 

\begin{lemma}\label{lemma-matrices}
	For all $\delta >0$ and all $N>1$, there exists $l > 1$ such that if $A_1,..., A_l$  are matrices in ${\mathrm{GL}}(2,\mathbb{R})$ and $v, w \in \mathbb{R}^2$ are two unit vectors verifying:
	\begin{align}\label{eq-matrices}
	\|A_i\cdots A_1(v)\|\geq \dfrac{1}{2} \|A_i\cdots A_1(w)\|; \ \ \text{and} \ \
	\|A_i\|,\|A_i^{-1}\|\leq N,
	\end{align}
	for every $1\leq i \leq l$. Then, there exist $R_i:V_i \to V_i$ rotation of angles smaller than $\delta$ such that $R_lA_l\cdots R_1 A_1(w)$ and $A_l\cdots A_1(v)$ are parallel.
\end{lemma}

Let us  make some comments about what we have done so far. Firstly, we have by Lemma~\ref{lemma-domination} that if $\dim \ker(Df_0)=\kappa$ then there exists a neighborhood $\mathcal{U}$ of $f_0$ and a neighborhood $U$ of $\mathrm{Cr}_{\kappa}(f_0)$ such that for every $f \in \mathcal{F}\cap \mathcal{U}$ and  every $(x_i)_i \in \Lambda_{f}$, we have that $\mathrm{Cr}_{\kappa}(f) \subseteq U$ and that the lack of domination property for the splitting $E\oplus F$ over $\Lambda_f$  only might occur away from $U$. Then, we can get $N \geq 1$ so that every $f \in \mathcal{U}$ satisfies:
\begin{align}\label{ineq:1}
\frac{1}{N}\leq \|Df_x|_{\Gamma}\|\leq N,\, \forall (x,\Gamma) \in \mathrm{Grass}_{\kappa}(TM) \ \ \text{with} \ \ x \notin U,
\end{align}
since $\ker(Df_x)$ for each $x \notin U$ has dimension strictly smaller $\kappa$ and  $Df$ on $M\backslash U$ is a continuous maps on compact set. In particular, $1/N\leq \|Df\mid_{E(x)}\|\leq N$ for each $x \notin U$ and each $f \in \mathcal{F}\cap \mathcal{U}$.

Recall that in the setting $\dim\ker(Df_0)=r<\kappa$, we can suppose that for every $f \in \mathcal{U}$ the following hold.
\begin{align}\label{ineq-2}
\dim \ker(Df)\leq r \ \ \text{and} \ \ \frac{1}{N}\leq \|Df_x|_{\Gamma}\|\leq N,\, \forall (x,\Gamma) \in \mathrm{Grass}_{\kappa}(TM).
\end{align}

We now fix $\ell_0\geq 1$ given by Lemma~\ref{lemma-matrices} for $\delta=\alpha$ and $N$ above. In sequel, we prove that the uniform domination property. In the first step, we assert that there is no large pieces of orbit without domination property. More concretely,
\begin{steps}
	For every $f \in \mathcal{F}\cap \mathcal{U}$ and every $(x_i)_i \in \Lambda_f$, there is an integer $1\leq j \leq \ell_0$ which depends of the orbit $(x_i)_i$ such that:
	\begin{align}\label{eq:step-1}
	\|Df^j\mid_{E(x_0)}\|\leq \frac{1}{2}\m(Df^j\mid_{F(x_0)}).
	\end{align}
\end{steps}

\begin{proof}[Proof of Step 1]
	The proof of this step will be by contradiction. That is, assume without loss of generality that for some $(x_i)_i \in \Lambda_f$,
	\begin{align}
	\|Df^j\mid_{E(x_0)}\|\geq \frac{1}{2}\m(Df^j\mid_{F(x_0)}).
	\end{align}
	for every $1\leq j \leq \ell_0$. Then, there exist two unit vectors $v_j \in E(x_0)$ and $w_j \in F(x_0)$ so that
	\begin{align}\label{eq-proof-2} \|Df^j(v_j)\|\geq \dfrac{1}{2}\|Df^j(w_j)\|,
	\end{align} for each $1\leq j \leq \ell_0$. We notice that $Df^j(v_j)\neq 0$ for every $1 \leq j \leq \ell_0$ since $Df^n(w_j)\neq 0$ for every $n\geq 1$. 
	
	It should also be note that, in the particular setting that $|\ker(Df_0)|=\kappa$, one has that $x_0, \dots, \, x_{\ell_0-1} \notin U$ by Lemma~\ref{lemma-domination}.
	
	Let us identify $T_{x_i}M$ with $\mathbb{R}^d$, for $0 \leq i \leq \ell_0$, and fix the notation $v=v_{\ell_0}$ and $w=w_{\ell_0}$. Take the family of matrices $A_1,...,A_m$ defined by $A_i=Df_{x_{i-1}}$ and the family of two-dimensional subspace $V_i=A_i\cdots A_1(V_0)$, where $V_0$ denote the subspace spanned by $v$ and $w$. Then, denote by $\tilde{A}_i:V_{i-1}\to V_i$ the restriction of $A_i$ to $V_{i-1}$ which is a isomorphism and, by \eqref{ineq:1} and \eqref{ineq-2}, verifies:
	\begin{align}
	\begin{split}
	&\max_{1\leq i\leq \ell_0}\{\|\tilde{A}_i\|,\|\tilde{A}^{-1}_i\|\}\leq N.
	\end{split}
	\end{align}
	This family $\tilde{A}_1,\dots,\tilde{A}_{\ell_0}$ satisfy (\ref{eq-matrices}). Therefore, there exist $\tilde{R}_{\ell_0}\dots, \tilde{R}_1$ rotations of angles less than $\theta_0$  such that
	$\tilde{R}_{\ell_0}\tilde{A}_{\ell_0}\cdots \tilde{R}_1\tilde{A}_1(w)$ and $\tilde{A}_{\ell_0}\cdots \tilde{A}_1(v)$ are parallel. Extending $\tilde{R}_i$ to the linear map $R_i$ on $\mathbb{R}^d$ that equals to identity on the orthogonal complement of $V_i$. Finally, $\Sigma=\{x_0,\dots,x_{\ell_0-1}\}$ and $L_i=R_i\circ A_i$, for $1\leq i \leq \ell_0$, verify the hypothesis of Lemma \ref{Franks-version} which is a contradiction. Therefore, one has that every $f \in \mathcal{F}\cap \mathcal{U}$ verifies \eqref{eq:step-1}.
\end{proof}

Finally, we can conclude the proof of uniform domination property.

\begin{steps}
	There exists $\ell \geq 1$ so that for every $f \in \mathcal{F}\cap \mathcal{U}$, the splitting $E\oplus F$ over $\Lambda_f$ satisfies $E\prec_{\ell} F$.
\end{steps}

\begin{proof}[Proof of Step 2]
	By Step 1, let $\ell \geq \ell_0$ be very large. Then, for all $(x_i)_i \in \Lambda_f$, there exists $1\leq j_0\leq \ell_0$ satisfying:
	\begin{align*}
	\|Df^{\ell}\mid_{E(x_0)}\|&\leq \|Df^{\ell-j_0}\mid_{E(x_{j_0})}\|\|Df^{j_0}\mid_{E(x_0)}\| \\
	&\leq \left(\frac{1}{2}\right) \|Df^{\ell-j_0}\mid_{E(x_{j_0})}\|m(Df^{j_0}\mid_{F(x_0)}).
	\end{align*}
	On the other hand, taking the point $(x_{j+j_0})_j$, so there exists $1 \leq j_1 \leq \ell_0$ such that:
	\begin{align*}
	\|Df^{\ell}\mid_{E(x_0)}\|&\leq \|Df^{\ell-(j_0+j_1)}\mid_{E(x_{j_0+j_1})}\|\|Df^{j_1}\mid_{E(x_{j_0})}\|\|Df^{j_0}\mid_{E(x_0)}\| \\
	&\leq \left(\frac{1}{2}\right)^2 \|Df^{\ell-(j_0+j_1)}\mid_{E(x_{j_1})}\|\m(Df^{j_1}\mid_{F(x_{j_0})})\m(Df^{j_0}\mid_{F(x_0)})\\
	&\leq \left(\frac{1}{2}\right)^2 \|Df^{\ell-(j_0+j_1)}\mid_{E(x_{j_1})}\|\m(Df^{j_0+j_1}\mid_{F(x_0)}).
	\end{align*}
	Repeating the process, we can find $n\geq 1$ so that $m_n=\ell-\sum_{i=0}^{n-1}j_i$ satisfies $1 \leq m_n \leq \ell_0$; and
	\begin{align*}
	\|Df^{\ell}\mid_{E(x_0)}\|&\leq\left(\frac{1}{2}\right)^n \|Df^{m_n}\mid_{E(x_{j_n})}\|\prod_{i=0}^{n}\m(Df^{j_i}\mid_{F(x_{j_i})}) \\
	&\leq \left(\frac{1}{2}\right)^n \|Df^{m_n}\mid_{E(x_{j_n})}\|\m(Df^{{\ell}-m_n}\mid_{F(x_0)})\\
	&\leq  \left(\frac{1}{2}\right)^n \dfrac{\max_{1\leq i\leq \ell_0}\{\|Df^i_x\|:x\in M \}}{\m(Df^{m_n}\mid_{F(x_{\ell-m_n})})} \m(Df^{m_n}\mid_{F(x_{\ell-m_n})})\m(Df^{{\ell}-m_n}\mid_{F(x_0)})\\		
	&\leq C_0 \left(\frac{1}{2}\right)^n\m(Df^{\ell}\mid_{F(x_0)})
	\end{align*}
	where $C_0$ is a constant satisfying:
	\begin{align*}
	C_0 \geq \dfrac{\max_{1\leq i\leq \ell_0}\{\|Df^i_x\|:x\in M \}}{\min_{1\leq j \leq \ell_0}\{\|Df^j\mid_{F(x_i)}\|: (x_i)_i \in \Lambda_f \}} >0.
	\end{align*}
	
	Note that  $C_0$ is well-defined and is uniform in $\mathcal{F}\cap \mathcal{U}$. Indeed, we can shrink $\mathcal{U}$ (if necessary) to obtain that $\dim\ker(Df^j) \leq \dim \ker(Df^j_0)\leq \kappa$ at $x$ and $\ker(Df^j)$ is close to $\ker(Df^j_0)$ on $T_xM$ for each $x \in M, \ \ 1\leq j\leq \ell_0,$ and  $f \in \mathcal{U}$. Then we assume that $\ker(Df^j_x)$ is contained in a cone $\mathscr{C}_{V}(x,\eta)$, where $V$ is the kernel of $Df_0^{\ell_0}$ at $x$ and $0<\eta<\alpha$, for each $x \in M$ and $1\leq j \leq \ell_0$. Since $\ker(Df^j_{x_i})\subseteq E(x_i)$, for every $1\leq j \leq \ell_0$, and $\varangle(E(x_i), F(x_i))\geq \alpha$ for every $(x_i)_i \in \Lambda_f$ and $f \in \mathcal{F}\cap \mathcal{U}$, we conclude that $F(x_i)$ is contained in the complement cone of $\mathscr{C}_{V}(x, \eta)$ and that is enough to get that $\min_{1\leq j \leq \ell_0}\{\|Df^j\mid_{F(x_i)}\|: (x_i)_i \in \Lambda_f \}$ is uniformly away from zero.
	
	
	Finally, we can choose $\ell>0$ large enough so that $n\geq 1$ satisfies $C_02^{-n}\leq 2^{-1}$ and so $E \prec_{\ell} F$. This concludes the uniform domination property.
\end{proof}

Therefore, we have just proved that every $f \in \mathcal{F}\cap \mathcal{U}$ admits $E\oplus F$ as a $(\alpha, \ell)$-dominated splitting. This concludes the proof of Theorem~\ref{thm B}. \qed

\begin{acknowledgement}
	The authors thank the referee for her/his appropriate comments and suggestions.
\end{acknowledgement}



\bibliographystyle{alpha}
\bibliography{Thesis-Biblio}

\begin{thebibliography}{{Zhu}19}

\bibitem[AB12]{AB}
A.~Avila and J.~Bochi.
\newblock Nonuniform hyperbolicity, global dominated splittings and generic
  properties of volume-preserving diffeomorphisms.
\newblock {\em Transactions of the American Mathematical Society},
  364(6):2883--2907, 2012.

\bibitem[AH94]{AH}
N.~Aoki and K.~Hiraide.
\newblock {\em Topological theory of dynamical systems}.
\newblock North-Holland Mathematical Library, 1st edition, 1994.

\bibitem[BD96]{BD}
C.~Bonatti and L.J. D{\'{\i}}az.
\newblock Persistent nonhyperbolic transitive diffeomorphisms.
\newblock {\em Ann. of Math. (2)}, 143(2):357--396, 1996.

\bibitem[BDP03]{BDP}
C.~Bonatti, L.~J. D{\'{\i}}az, and E.~Pujals.
\newblock A {$C^1$}-generic dichotomy for diffeomorphisms: weak forms of
  hyperbolicity or infinitely many sinks or sources.
\newblock {\em Ann. of Math. (2)}, 158(2):355--418, 2003.

\bibitem[BG09]{BG}
Jairo Bochi and Nicolas Gourmelon.
\newblock Some characterizations of domination.
\newblock {\em Mathematische Zeitschrift}, 262(713), July 2009.

\bibitem[BM15]{BM}
J.~Bochi and I.~D. Morris.
\newblock Continuity properties of the lower spectral radius.
\newblock {\em Proc. Lond. Math. Soc. (3)}, 110(2):477--509, 2015.

\bibitem[BPS19]{BPS}
J.~Bochi, R.~Potrie, and A.~Sambarino.
\newblock Anosov representations and dominated splittings.
\newblock {\em Jornal European Mathematical Society}, 21(11):3343\--3414, July
  2019.

\bibitem[BR13]{Berger-Rovella}
P.~Berger and A.~Rovella.
\newblock On the inverse limit stability of endomorphisms.
\newblock In {\em Annales de l'Institut Henri Poincare (C) Non Linear
  Analysis}, volume~30, pages 463--475. Elsevier, 2013.

\bibitem[BV00]{BV}
C.~Bonatti and M.~Viana.
\newblock Srb measures for partially hyperbolic systems whose central direction
  is mostly contracting.
\newblock {\em Israel J. Math.}, 115:157--193, 2000.

\bibitem[CP]{Crovisier-Potrie}
S.~Crovisier and R.~Potrie.
\newblock Introduction to partially hyperbolic dynamics.
\newblock {\em Notes for a minicourse in the School on Dynamical Systems 2015
  at ICTP}.

\bibitem[DPU99]{DPU}
L.~J. D{\'\i}az, E.~Pujals, and R.~Ures.
\newblock Partial hyperbolicity and robust transitivity.
\newblock {\em Acta Mathematica}, 183(1):1--43, 1999.

\bibitem[Fra71]{Franks}
J.~Franks.
\newblock Necessary conditions for stability of diffeomorphisms.
\newblock {\em Trans. Amer. Math. Soc.}, 158:301--308, 1971.

\bibitem[HJ12]{MatrixAnalysis}
Roger~A. Horn and Charles~R. Johnson.
\newblock {\em Matrix Analysis}.
\newblock Cambridge University Press, 2 edition, 2012.

\bibitem[ILP16]{ILP}
J.~Iglesias, C.~Lizana, and A.~Portela.
\newblock Robust transitivity for endomorphisms admitting critical points.
\newblock {\em Proc. Amer. Math. Soc.}, 144(3):1235--1250, 2016.

\bibitem[IP18]{IP}
Jorge Iglesias and Aldo Portela.
\newblock An example of a map which is $c^2$-robustly transitive but not
  $c^1$-robustly transitive.
\newblock {\em Colloquium Mathematicum}, 152, 03 2018.

\bibitem[LP13]{LP}
C.~Lizana and E.~Pujals.
\newblock Robust transitivity for endomorphisms.
\newblock {\em Ergodic Theory and Dynamical Systems}, 33(4):1082--1114, 2013.

\bibitem[LR17]{CW1}
C.~Lizana and W.~Ranter.
\newblock Topological obstructions for robustly transitive endomorphisms on
  surfaces.
\newblock {\em arXiv:1711.02218}, 2017.

\bibitem[LR19]{CW2}
C.~Lizana and W.~Ranter.
\newblock New classes of {$C^1$} robustly transitive maps with persistent
  critical points.
\newblock {\em arXiv:1902.06781}, 2019.

\bibitem[Ma{\~n}78]{Mane}
R.~Ma{\~n}{\'e}.
\newblock Contributions to the stability conjecture.
\newblock {\em Topology}, 17(4):383--396, November 1978.

\bibitem[Ma{\~n}82]{Mane-Closinglemma}
R.~Ma{\~n}{\'e}.
\newblock An ergodic closing lemma.
\newblock {\em Annals of Mathematics}, 116(3):503--540, 1982.

\bibitem[Mor20]{Morelli}
J.~C. Morelli.
\newblock An example of a {$\mathbb{T}^n$} endomorphism that is persistently
  singular and {$C^1$} robustly transitive.
\newblock {\em arXiv:2008.00911}, 2020.

\bibitem[MS74]{Milnor-Stasheff}
J.W. Milnor and J.D. Stasheff.
\newblock {\em Characteristic Classes}.
\newblock Annals of mathematics studies. Princeton University Press, 1974.

\bibitem[Pot12]{Potrie}
R.~Potrie.
\newblock {\em Partial Hyperbolicity and attracting regions in 3-dimensional
  manifolds}.
\newblock PhD thesis, PEDECIBA-Universidad de La Republica-Uruguay, 2012.

\bibitem[QTZ19]{QTZ}
A.~Quas, P.~Thieullen, and M.~Zarrabi.
\newblock Explicit bounds for separation between oseledets subspaces.
\newblock {\em Dynamical Systems}, 34(3):517\--560, feb 2019.

\bibitem[Shu71]{Shub-ex}
M.~Shub.
\newblock Topologically transitive diffeomorphisms of $\mathbb{T}^4$.
\newblock {\em Symposium on Differential Equations and Dynamical Systems,
  Springer Lecture Notesin Mathematics}, 206:39--40, 1971.

\bibitem[{Zhu}19]{Z}
Feng {Zhu}.
\newblock {Relatively dominated representations}.
\newblock {\em arXiv e-prints}, page arXiv:1912.13152, December 2019.

\end{thebibliography}

\end{document}